\documentclass[11pt,reqno]{amsart}

\usepackage[centertableaux]{ytableau}
\usepackage{tikz}
\usepackage{graphicx}
\usepackage[hyphens]{url}
\usepackage{hyperref}
\usepackage[hyphenbreaks]{breakurl}
\usepackage[margin=3.5cm]{geometry}
\usepackage{mathtools}
\usepackage[labelfont=normalfont]{subcaption}
\usepackage{comment}
\usepackage{amssymb}
\usepackage{musicography}	

\newtheorem{thm}{Theorem}[section]
\newtheorem{lem}[thm]{Lemma}
\newtheorem{cor}[thm]{Corollary}

\theoremstyle{definition}

\let\sp\relax
\DeclareMathOperator{\sp}{sp}
\DeclareMathOperator{\so}{so}
\let\o\relax
\DeclareMathOperator{\o}{o}
\DeclareMathOperator{\SSYT}{SSYT}
\DeclareMathOperator{\SPT}{SPT}
\DeclareMathOperator{\SOT}{SOT}
\DeclareMathOperator{\OT}{OT}
\DeclareMathOperator{\sgn}{sgn}

\newcommand{\abs}[1]{\lvert#1\rvert}

\newcommand{\la}{\lambda}
\newcommand{\La}{\Lambda}
\newcommand{\x}{\mathbf{x}}
\newcommand{\X}{\mathbf{X}}

\newcommand{\ov}[1]{\overline{#1}}
\newcommand{\wh}[1]{\widehat{#1}}
\newcommand{\defn}[1]{\emph{\color{cyan!70!black} #1}}

\renewcommand{\leq}{\leqslant}
\renewcommand{\geq}{\geqslant}

\newcommand\Square[1]{+(-#1,-#1) rectangle +(#1,#1)}

\DeclareFontFamily{U}{mathx}{}
\DeclareFontShape{U}{mathx}{m}{n}{<-> mathx10}{}
\DeclareSymbolFont{mathx}{U}{mathx}{m}{n}
\DeclareMathAccent{\wc}{0}{mathx}{"71}

\begin{document}
\title[Skew characters through lattice paths]{Skew symplectic and orthogonal characters through lattice paths}
\author[S.~P.~Albion]{Seamus P.~Albion}
\author[I.~Fischer]{Ilse Fischer}
\author[H.~H\"ongesberg]{Hans H\"ongesberg}
\author[F.~Schreier-Aigner]{Florian Schreier-Aigner}
\date{}
\thanks{I.F., H.H., and F.S.-A.\ acknowledge the financial support from the Austrian Science Foundation FWF, grant P34931.}
\begin{abstract}
The skew Schur functions admit many determinantal expressions. Chief among
them are the (dual) Jacobi--Trudi formula and the Lascoux--Pragacz formula, the latter being 
a skew analogue of the Giambelli identity.
Comparatively, the skew characters of the symplectic and orthogonal groups, also
known as the skew symplectic and orthogonal Schur functions, have
received less attention in this direction.
We establish analogues of the dual Jacobi--Trudi and Lascoux--Pragacz formulae 
for these characters. Our approach is entirely combinatorial, being based on 
lattice path descriptions of the tableaux models of Koike and Terada. 
Ordinary Jacobi--Trudi formulae 
are then derived in an algebraic manner from their duals.
\end{abstract}

\maketitle
\section{Introduction}
\label{sec:intro}
The classical groups, a term coined by Weyl, are the general linear groups
over the real numbers, complex numbers and quaternions and certain subgroups thereof.
We are concerned with the complex general linear, symplectic and 
orthogonal groups, which we write as
$\mathrm{GL}(n,\mathbb{C})$, $\mathrm{Sp}(2n,\mathbb{C})$ and
$\mathrm{O}(n,\mathbb{C})$ respectively.
Each of these groups carry families of irreducible representations indexed
by partitions.
In the case of $\mathrm{GL}(n,\mathbb{C})$, these representations are 
precisely the irreducible polynomial representations whose characters 
are the Schur polynomials, which are symmetric polynomials
in $n$ variables.
For the symplectic and orthogonal groups, the irreducible characters in question
are symmetric Laurent polynomials in the variables 
$x_1,x_1^{-1},\dots,x_n,x_n^{-1}$
(also known as $\mathrm{BC}_n$-symmetric polynomials).
These characters sometimes go by the name of symplectic and orthogonal 
Schur polynomials, but we will refer to them simply as the {symplectic and
orthogonal characters.}

The Schur polynomials have several different determinantal expressions.
Among them are the {Jacobi--Trudi formula} and its dual, the 
{N\"agelsbach--Kostka identity}, which express the Schur polynomial as
(isobaric) determinants in the complete homogeneous or elementary symmetric
functions respectively.
There is also the {Giambelli formula}, which expresses the Schur
polynomial as a determinant of Schur polynomials indexed by hook-shaped 
Young diagrams.
These two types of determinantal expressions have analogues
for {skew Schur polynomials}, the skew version of the Giambelli formula being due 
to Lascoux and Pragacz \cite{LP84,LP88}.
Analogously, the symplectic and orthogonal characters were given 
Jacobi--Trudi-type expressions by Weyl \cite[Theorems~7.8.E \& 7.9.A]{Weyl39},
there expressed in terms of complete homogeneous or elementary symmetric functions with 
variables $x_1,x_1^{-1},\dots,x_n,x_n^{-1}$.
They also have Giambelli formulae, being first proved by 
Abramsky, Jahn and King \cite{AJK73}, the structure of which is
identical to the Schur case.

From a combinatorial point of view, the Schur polynomials may be defined 
as a weighted sum over {semistandard Young tableaux} or, equivalently, 
Gelfand--Tsetlin patterns. This extends easily to the skew case.
By interpreting these tableaux as families of nonintersecting lattice paths,
Gessel and Viennot \cite{GesVie89} provided a beautiful proof 
of the skew Jacobi--Trudi formulae using what is now known as the
Lindstr\"om--Gessel--Viennot lemma.
Stembridge then applied this approach to the Giambelli and Lascoux--Pragacz formulae
\cite[\S9]{Ste90}.

In \cite{FulKra97}, Fulmek and Krattenthaler set out to provide similar lattice
path proofs of the symplectic and orthogonal Jacobi--Trudi formulae and their duals
 as well as the Giambelli formulae in the straight case.
For the symplectic identities they use the tableaux of King \cite{King76}.
In the orthogonal case there are several different tableaux models, and of
these Fulmek and Krattenthaler exploit the tableaux of 
Proctor, Sundaram as well as of King and Welsh \cite{KinWel93,Pro94,Sun90b}.
The main tools in their proofs are the Lindstr\"om--Gessel--Viennot lemma
and a modified reflection principle.
They succeeded in proving the dual Jacobi--Trudi formulae for the 
symplectic and orthogonal characters in this way, and pass to the ordinary
formulae by the dual paths of Gessel and Viennot.
For the Giambelli identities, they prove the symplectic and odd orthogonal
cases.
However, they were unable to obtain a lattice path proof of the even 
orthogonal Giambelli formula. (This was due to the even orthogonal 
Sundaram-type tableaux they introduce not having an appropriate weight 
function; see the discussions in \cite[\S3.6]{BreKraWar16} and \cite[\S8]{FulKra97}.)

Compared to the skew Schur polynomials, skew analogues of the symplectic and
orthogonal characters have received little attention.
Indeed, only very recently did Jing, Li and Wang obtain Jacobi--Trudi
formulae for these characters \cite[Propositions~3.2--3.3]{skew2022},
and there only in terms of the complete symmetric functions.
The goal of the present paper is to provide dual Jacobi--Trudi-type formulae
for the skew symplectic and orthogonal characters.
We accomplish this in a purely combinatorial way, 
using an approach that is somewhat in the vein of the one used by 
Fulmek and Krattenthaler with an extension to the skew setting.
In addition, we also produce Lascoux--Pragacz-type skew analogues of the
Giambelli formulae for these characters.
What facilitates these more general formulae are the tableaux models for
skew symplectic and orthogonal characters due to Koike and Terada
\cite{KoiTer90}.
These tableaux have already proved combinatorially useful in proving 
factorisation theorems for skew symplectic and orthogonal characters in work 
of Ayyer and the second named author \cite{AyyerFischer20}.
Remarkably, our proofs are simpler than those of Fulmek and Krattenthaler,
which further emphasises that the tableaux of Koike and Terada are the
better tool in the combinatorial setting.

We begin in the next section by providing definitions and statements of our
results. In Section~\ref{sec:tableaux} we introduce the tableaux of Koike and 
Terada as well as the corresponding families of nonintersecting lattice paths. 
In the following Section~\ref{sec:dualJacobiTrudi} we give proofs
of the dual Jacobi--Trudi formulae.
In order to relate these to the results of Jing, Li and Wang, we use the 
standard algebraic approach for proving the equivalence of the ordinary
Jacobi--Trudi formula and its dual in Section~\ref{sec:JacobiTrudi}.
Following this, Section~\ref{sec:Giambelli} contains the proofs of the 
Lascoux--Pragacz-type skew analogues of the Giambelli formulae.

\section{Definitions and main results}
\label{sec:definitions}
A \defn{partition} is a weakly decreasing sequence of positive integers 
$\la=(\la_1,\ldots, \la_k)$.
We call the $\la_i$ its \defn{parts}, $l(\mu) \coloneqq k$  its \defn{length} and $\abs{\la} \coloneqq \sum_{i=1}^{k} \la_i$ its \defn{size}.
The \defn{Young diagram} of a partition $\la$ is a collection
of left-justified boxes (\defn{cells}) with $\la_i$ cells in the $i$-th row 
from the top.
For the rest of the paper, we do not distinguish between a partition and its associated Young diagram. 
The \defn{conjugate} $\la^\prime=(\la_1^\prime,\ldots,\la_{m}^\prime)$ of 
$\la$
is the partition where $\la_i^\prime$ is the number of boxes in the
$i$-th column of the Young diagram of $\la$, counted from the left.
For two partitions $\la, \mu$, we say that $\mu$ is \defn{contained} in $\la$, 
denoted by $\mu \subseteq \la$, if the Young diagram of $\mu$ can be obtained 
from the Young diagram of $\la$ by removing boxes. In this case, we
denote by $\la/\mu$ the \defn{skew shaped Young diagram} (or \defn{skew shape} 
for short) obtained by removing all boxes of the Young diagram of $\mu$ from 
the one of $\la$. The size of $\la/\mu$ is the number of boxes in $\la/\mu$ 
and denoted by $|\la/\mu|$.
We will sometimes write $(n^m)$ for the rectangular partition with $m$ parts
equal to $n$.

We are interested in irreducible characters of the classical Lie groups 
$\mathrm{GL}(n,\mathbb{C})$, $\mathrm{Sp}(2n,\mathbb{C})$ and 
$\mathrm{O}(n,\mathbb{C})$ indexed by partitions.
The definitions of these objects may be found, for instance, in
\cite[\S24]{FulHar91} or \cite[Appendix B]{Pro94},
and we only cover the essentials needed to state our results.
The characters we are interested in are all symmetric or 
$\mathrm{BC}_n$-symmetric polynomials with variables $\x\coloneqq(x_1,\dots,x_n)$, 
the rings of which we denote by $\La_n$ and $\La_n^{\mathrm{BC}}$ respectively.
For $\mathrm{GL}(n,\mathbb{C})$, the characters of the irreducible polynomial
representations are the \defn{Schur polynomials} $s_\la(\x)$ where 
$\la$ runs over all partitions of length at most $n$.
They are easily computed, for $l(\la)\leq n$, by
\begin{equation}\label{Eq_schur-def}
s_\la(\x) \coloneqq \frac{\det_{1\leq i,j\leq n}(x_i^{\la_j+n-j})}
{\det_{1\leq i,j\leq n}(x_i^{n-j})}, 
\end{equation}
setting $\la_j=0$ for $j > l(\la)$. In fact, this convention is used 
throughout the whole paper. The formula is a special case of the Weyl character formula.
If $l(\la)>n$ then $s_\la(\x) \coloneqq 0$.
From the definition it is clear that these are in fact symmetric polynomials of 
homogeneous degree $\abs{\la}$.
When $\la$ is a single row or column of $r$ boxes then the Schur polynomials
reduce to the \defn{complete homogeneous symmetric polynomials} $h_r$ and
the \defn{elementary symmetric polynomials} $e_r$ in $\x$ respectively.

For the symplectic group $\mathrm{Sp}(2n,\mathbb{C})$ we have irreducible
characters given by the Weyl formula
\[
\sp_{\la}(\x) \coloneqq
\frac{\det_{1\leq i,j\leq n}(x_i^{\la_j+n-j+1}-x_i^{-(\la_j+n-j+1)})}
{\det_{1\leq i,j\leq n}(x_i^{n-j+1}-x_i^{-(n-j+1)})},
\]
where again $l(\la)\leq n$.
The orthogonal case is little more delicate. 
For a statement $P$ we write $[P]$ for the \defn{Iverson bracket}:
$[P]=1$ if $P$ is true and $[P]=0$ otherwise. 
Then the even orthogonal group $\mathrm{O}(2n,\mathbb{C})$ has irreducible 
characters
\[
\o_{\la}(\x) \coloneqq 2^{[\la_n\not=0]}
\frac{\det_{1\leq i,j\leq n}(x_i^{\la_j+n-j}+x_i^{-(\la_j+n-j)})}
{\det_{1\leq i,j\leq n}(x_i^{n-j}+x_i^{-(n-j)})}.
\]
If $\la_n=0$ then this is also the character of the irreducible
representation of the special orthogonal group $\mathrm{SO}(2n,\mathbb{C})$
corresponding to $\la$.
However, if $\la_n\neq 0$ then the above splits into a sum of two
irreducible characters of $\mathrm{SO}(2n,\mathbb{C})$, one indexed by
$\la$ and the other by $(\la_1,\dots,\la_{n-1},-\la_n)$, which is of 
course not a partition.
In the odd orthogonal case $\mathrm{O}(2n+1,\mathbb{C})$ there is no such
distinction, and therefore we will label the irreducible characters
of $\mathrm{O}(2n+1,\mathbb{C})$ (or $\mathrm{SO}(2n+1,\mathbb{C})$) by
\[
\so_{\la}(\x) \coloneqq
\frac{\det_{1\leq i,j\leq n}(x_i^{\la_j+n-j+1/2}-x_i^{-(\la_j+n-j+1/2)})}
{\det_{1\leq i,j\leq n}(x_i^{n-j+1/2}-x_i^{-(n-j+1/2)})}.
\]
We will refer to these two sets of orthogonal characters as 
even orthogonal characters and odd orthogonal characters respectively.

Each of the above four families of characters have skew variants.
In Section~\ref{sec:tableaux} we will explicitly define each of these families 
in terms of skew tableaux, however for now let us explain where they come from.
We say two partitions \defn{interlace}, written $\mu\preccurlyeq\la$, if 
$\mu\subseteq\la$ and
\[
\la_1\geq\mu_1\geq\la_2\geq\mu_2\geq\cdots.
\]
One of the fundamental properties of the Schur polynomials is the 
\defn{branching rule} \cite[p.~72]{Macdonald}
\[
s_\la(x_1,\dots,x_n)
=\sum_{\mu\preccurlyeq\la}x_n^{\abs{\la/\mu}}s_\mu(x_1,\dots,x_{n-1}).
\]
Since $\la$ and $\mu$ interlace, the skew shape $\la/\mu$ has no
two boxes in the same column.
Iterating the branching rule $n$ times naturally leads to the notion of 
semistandard Young tableaux.
Alternatively, one could iterate only $k$ times for $1\leq k\leq n-1$.
The coefficient of $s_\mu(x_1,\dots,x_{n-k})$ in this
expansion is the skew Schur polynomial $s_{\la/\mu}(x_{n-k+1},\dots,x_n)$.
In other words, we have the more general branching rule
\[
s_\la(x_1,\dots,x_n)=\sum_{\mu\subseteq\la}
s_\mu(x_1,\dots,x_{n-k})s_{\la/\mu}(x_{n-k+1},\dots,x_n).
\]
In representation-theoretic terms, the branching rule 
describes the restriction of the irreducible representation indexed by $\la$ to 
the subgroup $\mathrm{GL}(n-1,\mathbb{C}) \times \mathrm{GL}(1,\mathbb{C})$,
or more generally to $\mathrm{GL}(n-k,\mathbb{C}) \times \mathrm{GL}(k,\mathbb{C})$.

Koike and Terada \cite{KoiTer90} carried out this same procedure for the
symplectic and orthogonal groups.
They use the branching rules of Zhelobenko \cite{Zhe62}
to define skew analogues of the symplectic and orthogonal characters.
In their most general form these characters depend on a skew shape
$\la/\mu$ and integers $n,m$ such that $l(\mu)\leq m$ and
$l(\la)\leq n+m$.
Following our notation from above, we denote these by
$\sp_{\la/\mu}^m,\so_{\la/\mu}^m$ and $\o_{\la/\mu}^m$.
These objects are, like their non-skew variants, symmetric Laurent polynomials
in $x_1,x_1^{-1}\dots,x_n,x_n^{-1}$.
While we give a combinatorial definition of these characters below, we should
mention that Jing, Li and Wang have given alternative definitions in terms
of vertex operators \cite{skew2022}.

The main goal of this paper is to prove, combinatorially, two types of 
determinantal formulae for $\sp_{\la/\mu}^m$, $\so_{\la/\mu}^m$ and
$\o_{\la/\mu}^m$: (i) dual Jacobi--Trudi-type formulae and (ii) 
Giambelli--Lascoux--Pragacz-type formulae.
The Jacobi--Trudi-type formulae are then derived from their duals in an 
algebraic manner.

For $\mu \subseteq \la$, $l(\lambda) \leq n$ and $\lambda_1 \leq N$, and the Jacobi--Trudi identity and its dual are
\begin{equation}
s_{\la/\mu}(\x)=\det_{1\leq i,j\leq n}(h_{\la_i-\mu_j-i+j}(\x))
=\det_{1\leq i,j\leq N}(e_{\la_i'-\mu_j'-i+j}(\x)), 
\end{equation}
where we remind the reader that $\x=(x_1,\dots,x_n)$ for a non-negative integer $n$ that is 
fixed throughout the paper.
Here we have stated this as an identity for the Schur polynomials, but it also
holds in the ring of symmetric functions on a countable set of variables,
in which case the $s_{\la/\mu}$ are the \defn{skew Schur functions}.

We set $\x^{\pm}\coloneqq(x_1,x_1^{-1},\dots,x_n,x_n^{-1})$
and are now ready to state our theorems.

\begin{thm} 
\label{thm:dual} 
Let $m, n, N$ be non-negative integers and $\lambda, \mu$ partitions such that
$\mu\subseteq\lambda$, $l(\mu)\leq m$, $l(\la) \leq n+m$ and $\la_1 \leq N$.
Then
\begin{subequations}\label{Eq_dual}
\begin{align} 
\label{dual_sp}
\sp^m_{\lambda/\mu} (\x)&= \det_{1 \leq i,j \leq N} \big( e_{\lambda'_i -\mu'_j-i+j}(\x^\pm) - e_{\lambda'_i + \mu'_j-i-j-2m}(\x^\pm) \big), \\
\label{dual_so}
\so^m_{\lambda/\mu} (\x)&= \det_{1 \leq i,j \leq N} \big( e_{\lambda'_i -\mu'_j-i+j}(\x^\pm) + e_{\lambda'_i + \mu'_j-i-j-2m+1}(\x^\pm) \big), \\
\label{dual_o}
\o^m_{\lambda/\mu} (\x) &=\frac{1}{2^{[m=l(\mu)]}} \det_{1 \leq i,j \leq N} \big( e_{\lambda'_i -\mu'_j-i+j}(\x^\pm) + e_{\lambda'_i + \mu'_j-i-j-2m+2}(\x^\pm) \big).
\end{align} 
\end{subequations}
\end{thm} 

A main contribution of our work is to provide, as we believe, enlightening combinatorial explanations of these 
formulae. By using an algebraic approach, we can dualise these three identities and
obtain the following ordinary Jacobi--Trudi-type formulae for these characters.

\begin{thm}\label{thm:hformulae}
Let $m, n, N$ be non-negative integers and $\la,\mu$ partitions such that
$\mu\subseteq\la$, $l(\mu)\leq m$, $l(\la)\leq n+m$ and $l(\la)\leq N$.
Then
\begin{subequations}\label{Eq_hformulae}
\begin{align}\label{Eq_sph}
\sp^m_{\la/\mu}(\x)&=
\det_{1\leq i,j\leq N}\big(h_{\la_i-\mu_j-i+j}(\x^\pm)
+[j>m+1]h_{\la_i-i-j+2m+2}(\x^\pm)\big),\\
\so^m_{\la/\mu}(\x)&=
\det_{1\leq i,j\leq N}\big(h_{\la_i-\mu_j-i+j}(\x^\pm)
+[j>m]h_{\la_i-i-j+2m+1}(\x^\pm)\big), \label{Eq_soh}\\
\o^m_{\la/\mu}(\x)&=
\det_{1\leq i,j\leq N}\big(h_{\la_i-\mu_j-i+j}(\x^\pm)
-[j>m]h_{\la_i-i-j+2m}(\x^\pm)\big). \label{Eq_oh}
\end{align}
\end{subequations}
\end{thm}
For $\mu$ empty and $m=0$,
these formulae are due to Weyl
\cite[Theorems~7.8.E \& 7.9.A]{Weyl39}. 
The identities \eqref{Eq_sph} and \eqref{Eq_oh} as stated were recently 
obtained by Jing, Li and Wang \cite[Propositions~3.2--3.3]{skew2022}.

Our combinatorial approach also admits the derivation of Giambelli-type formulae, and, in 
order to formulate them, we need the \defn{Frobenius notation} of a partition.
For a partition $\lambda$, let $(p,p)$ be the diagonal cell with maximal $p$ 
which is still contained in the Young diagram of $\lambda$ 
($p$ is the size of the \defn{Durfee square}). 
For $1 \leq i \leq p$, let $\alpha_i$ be the number of cells right of 
$(i,i)$ in the same row and $\beta_i$ be the number of cells below $(i,i)$ in 
the same column.
We then write $\lambda=(\alpha_1,\dots,\alpha_p \vert \beta_1,\dots,\beta_p)$.
Using Frobenius notation for $\lambda$ and $\mu$, that is, 
$\lambda=(\alpha_1,\dots,\alpha_p \vert \beta_1,\dots,\beta_p)$ and 
$\mu=(\gamma_1,\dots,\gamma_q \vert \delta_1,\dots,\delta_q)$, the 
\defn{Lascoux--Pragacz formula} reads as
\begin{equation}\label{eq:ordinarySkewGiambelli}
s_{\lambda/\mu} (\x) = (-1)^q \det \begin{pmatrix}
			\left( s_{\left( \alpha_i \vert \beta_j \right)} (\x) \right)_{1 \leq i,j \leq p} & \left( h_{\alpha_i - \gamma_j} (\x) \right)_{\substack{1 \leq i \leq p,\\1 \leq j \leq q}}\\
			\left( e_{\beta_j - \delta_i} (\x) \right)_{\substack{1 \leq i \leq q,\\1 \leq j \leq p}} & 0
		\end{pmatrix}. 
\end{equation}
This was shown by Lascoux and Pragacz in \cite{LP84,LP88}.
The $\mu$ empty case is much older, being due to Giambelli 
\cite{Giambelli1903}, and is therefore known as the \defn{Giambelli identity}.
We provide the following analogues.

\begingroup
\allowdisplaybreaks
\begin{thm}\label{thm:giambelli}
	Let $\lambda=(\alpha_1,\dots,\alpha_p \vert \beta_1,\dots,\beta_p)$ and $\mu=(\gamma_1,\dots,\gamma_q \vert \delta_1,\dots,\delta_q)$ be two partitions such that $\mu \subseteq \lambda$, $l(\mu)\leq m$ and $l(\la)\leq n+m$. 
Then
\begin{subequations}\label{Eq_giambelli}
\begin{align}
\sp^m_{\lambda/\mu}(\x) &= (-1)^q \det \begin{pmatrix}
\big(\sp^m_{(\alpha_i \vert \beta_j)}(\x)\big)_{1 \leq i,j \leq p} & 
\big(\sp_{(\alpha_i)/(\gamma_j)}^m(\x)
\big)_{\substack{1 \leq i \leq p,\\1 \leq j \leq q}}\\
\big(\sp_{(1^{\beta_j+1})/(1^{\delta_i+1})}^m(\x)
\big)_{\substack{1 \leq i \leq q,\\1 \leq j \leq p}} & 0
\end{pmatrix},\label{Eq_spgiambelli} \\[2mm]
		\so^m_{\lambda/\mu}(\x) &= (-1)^q \det \begin{pmatrix}
\big( \so^m_{(\alpha_i \vert \beta_j)}(\x)\big)_{1 \leq i,j \leq p} & 
\big(\so_{(\alpha_i)/(\gamma_j)}^m(\x)
\big)_{\substack{1 \leq i \leq p,\\1 \leq j \leq q}}\\
\big(\so_{(1^{\beta_j+1})/(1^{\delta_i+1})}^m(\x)
\big)_{\substack{1 \leq i \leq q,\\1 \leq j \leq p}} & 0
		\end{pmatrix}, \label{Eq_sogiambelli}\\[2mm]
\o^m_{\lambda/\mu}(\x) &= 
(-1)^q\det \begin{pmatrix}
\big(\o^m_{(\alpha_i \vert \beta_j)}(\x) \big)_{1 \leq i,j \leq p} & 
\big(\o_{(\alpha_i)/(\gamma_j)}^m(\x)
\big)_{\substack{1 \leq i \leq p,\\1 \leq j \leq q}}\\
\big(\o_{(1^{\beta_j+1})/(1^{\delta_i+1})}^m(\x)
\big)_{\substack{1 \leq i \leq q,\\1 \leq j \leq p}} & 0
		\end{pmatrix}.\label{Eq_ogiambelli}
\end{align} 
\end{subequations}
\end{thm}
\endgroup
As mentioned in the introduction, for $\mu$ empty these formulae reduce to
those of Abramsky, Jahn and King \cite{AJK73}.
Like their formulae, ours have the exact same structure as in the Schur case
since $s_{(1^{\beta_j+1})/(1^{\delta_i+1})}(\x)=e_{\beta_j-\delta_i}(\x)$ and
$s_{(\alpha_i)/(\gamma_j)}(\x)=h_{\alpha_i-\gamma_j}(\x)$.

Weyl's Jacobi--Trudi formulae for the symplectic and orthogonal characters 
are used by Koike and Terada to define the \defn{universal characters} for 
these groups \cite{KoiTer87}.
This is achieved by ``forgetting'' the variables $\x^\pm$ in
either \eqref{Eq_dual} and \eqref{Eq_hformulae} and then treating the
determinants as polynomials in the $e_r$ or $h_r$ respectively,
and thus as elements of the ring of symmetric functions at an arbitrary 
alphabet.
By specialising the arbitrary alphabet to $\x^\pm$ the actual characters
are recovered.
Many identities between the general linear characters (Schur functions) and
the orthogonal and symplectic characters are derived by Koike and Terada 
using only the universal characters.
It would be interesting to investigate whether using \eqref{Eq_dual} and 
\eqref{Eq_hformulae} to define skew analogues of the universal characters
leads to similar ``universal'' proofs of identities such as, for example,
 branching rules.

Note that Hamel \cite{Ham97} has given determinantal formulae 
for skew analogues of the symplectic and odd orthogonal characters using 
outside decompositions as introduced in \cite{HG95}.
However, the skew tableaux she defines are different to those of 
Koike and Terada, and so are the associated skew characters.

\section{Tableaux and lattice paths}
\label{sec:tableaux}
In this section, we introduce the underlying combinatorial models for $s_{\la/\mu}, \sp^m_{\la/\mu}, \so^m_{\la/\mu}$ and  $\o^m_{\la/\mu}$
in terms of tableaux.
There are several possibilities for the models underlying $\so^m_{\la/\mu}$ and  $\o^m_{\la/\mu}$, and we choose those defined by Koike and Terada in \cite{KoiTer90}. In the non-skew case, all of the various combinatorial models and their equivalences may be found in \cite{Pro94} (also see \cite{CamSto13}). We then also introduce the corresponding families of non-intersecting lattice paths.

\subsection{Semistandard Young tableaux}
\label{sec:SSYTs}

Let $\mu \subseteq \la$ be two partitions. A \defn{semistandard Young tableau} of shape $\la / \mu$ is a filling of the cells of the Young diagram $\la/\mu$ with positive integers such that the entries increase weakly along rows and strictly down columns; see Figure~\ref{fig:SSYT} for an example. We denote by $\SSYT_{\la/\mu}^n$ the set of semistandard Young tableaux of shape $\la/\mu$ and maximal entry $n$. The weight $\x^T$ of a semistandard Young tableau $T$ is defined as the monomial
\[
\x^T = x_1^{\#\text{ of 1's in $T$}}\cdots x_n^{\#\text{ of $n$'s in $T$}}.
\]
The \defn{skew Schur polynomial} $s_{\la/\mu}(\x)$ corresponding to the shape $\la/\mu$ is defined as the multivariate generating function of semistandard Young tableaux of shape $\la/\mu$, i.e.,
\[
s_{\la/\mu}(\x) = \sum_{T \in \SSYT_{\la/\mu}^n}\x^T.
\]

Our path models depend on arbitrary integers $N \geq \lambda_1$ and $n \geq l(\la)$. 
A \defn{family of Schur paths} associated with the shape $\la/\mu$ is a family of $N$ non-intersecting lattice paths with starting points $S_i=(\mu_i^\prime-i+1,2 l(\mu)-\mu_i^\prime+i-1)$, end points $E_j=(\lambda_j^\prime-j+1,n-\la_j^\prime+j-1+2l(\mu))$ for $1 \leq i,j \leq N$, where also here we set $\la_j^\prime=0$ and $\mu_i^\prime=0$ for $j>l(\la^\prime)$
and $i>l(\mu')$, and with step set $\{(1,0),(0,1)\}$. The weight of a family of paths is the product of the weights of its steps, where the $i$-th step of a path has weight $x_i$ if it is horizontal and $1$ otherwise. Phrased differently, the weight of a horizontal step starting at $(a,b)$ has weight $x_{a+b-2 l(\mu)+1}$ (this is called the \defn{$e$-labelling}).

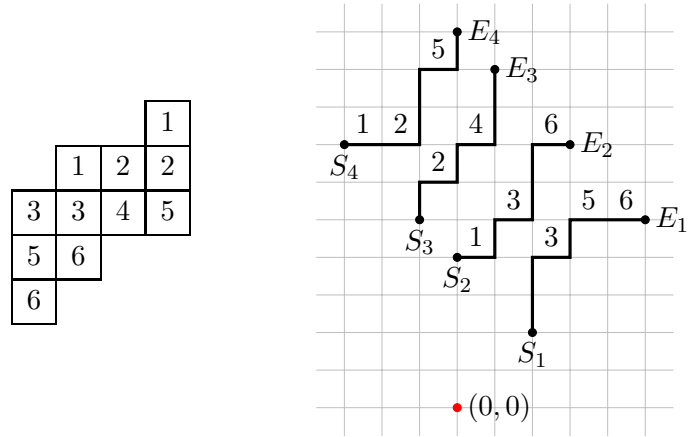
\begin{figure}[htb]
	\centering
	\begin{ytableau}
	\none & \none & \none & 1\\
	\none & 1 & 2 & 2\\
	 3 & 3 & 4 & 5\\
	 5 & 6 \\
	 6 
	\end{ytableau}
	\qquad\qquad
	\begin{tikzpicture}[scale=.5,baseline=(current bounding box.center)]
	\draw[help lines,step=1cm, lightgray] (-3.75,-0.75) grid (5.75,10.75);
	\filldraw[red] (0,0) circle (3pt) node[right]{\color{black}$(0,0)$};
	\filldraw[black] (2,2) node[below]{\color{black}$S_1$} circle (3pt);
	\filldraw[black] (0,4) node[below]{\color{black}$S_2$} circle (3pt);
	\filldraw[black] (-1,5) node[below]{\color{black}$S_3$} circle (3pt);
	\filldraw[black] (-3,7) node[below]{\color{black}$S_4$} circle (3pt);
	\filldraw[black] (5,5) node[right]{\color{black}$E_1$} circle (3pt);
	\filldraw[black] (3,7) node[right]{\color{black}$E_2$} circle (3pt);
	\filldraw[black] (1,9) node[right]{\color{black}$E_3$} circle (3pt);
	\filldraw[black] (0,10) node[right]{\color{black}$E_4$} circle (3pt);
	\draw[very thick] (2,2) -- (2,4) --node[above]{$3$} (3,4) -- (3,5) --node[above]{$5$}  (4,5) --node[above]{$6$}  (5,5);
	\draw[very thick] (0,4) --node[above]{$1$} (1,4) -- (1,5) --node[above]{$3$} (2,5) -- (2,7) --node[above]{$6$} (3,7);
	\draw[very thick] (-1,5) -- (-1,6) --node[above]{$2$} (0,6) -- (0,7) --node[above]{$4$} (1,7) -- (1,8) -- (1,9);
	\draw[very thick] (-3,7) --node[above]{$1$} (-2,7) --node[above]{$2$} (-1,7) -- (-1,9) --node[above]{$5$} (0,9) -- (0,10);
	\end{tikzpicture}
	\caption{A semistandard Young tableau of shape $(4,4,4,2,1)/(3,1)$ (left) and the corresponding family of Schur paths (right) for $N=4$ and $n=6$.}
	\label{fig:SSYT}
\end{figure}

The bijection between semistandard Young tableaux $T$ of shape $\la/\mu$ and families of Schur paths associated with the shape $\la/\mu$ is as follows. The path starting at $S_i$ corresponds to the $i$-th column of $T$ and the $j$-th step in the path is a horizontal step if and only if $j$ appears as a filling in the $i$-th column. See Figure~\ref{fig:SSYT} for an example.

\subsection{Skew \texorpdfstring{$(n,m)$}{(n,m)}-symplectic semistandard Young tableaux}
\label{sec:spYT}
Let $\mu \subseteq \lambda$ be two partitions and $n,m$ integers with 
$l(\mu)\leq m$ and $l(\la)\leq n+m$.
A \defn{skew $(n,m)$-symplectic semistandard Young tableau} (or \defn{skew $(n,m)$-symplectic tableau}) of skew shape $\la/\mu$ is a semistandard Young tableau of skew shape $\la/\mu$  which is filled by 
\[
\ov{1} < 1 < \ov{2} < 2 < \cdots < \ov{n} < n,
\]
and satisfies the
\begin{itemize}
\item \defn{$m$-symplectic condition}: the entries in row $m+i$ are at least $\ov{i}$.
\end{itemize}
We denote by $\SPT^{(n,m)}_{\la/\mu}$ the set of skew $(n,m)$-symplectic tableaux of shape $\la/\mu$. For a skew $(n,m)$-symplectic tableau $T$, the weight $\x^T$ is defined as
\[
\x^T =\prod_{i=1}^n x_i^{(\#\text{ of  $i$'s in $T$}) -(\#\text{ of $\ov{i}$'s in $T$})}.
\]
The \defn{skew $m$-symplectic character} $\sp^m_{\la/\mu}(\x)$ is defined as the multivariate generating function of skew $(n,m)$-symplectic tableaux of shape $\la/\mu$:
\[
\sp^m_{\la/\mu}(\x) = \sum_{T \in \SPT^{(n,m)}_{\la/\mu}}\x^T.
\]

For our path model, we again fix an integer $N \geq \lambda_1$.
A \defn{family of $(n,m)$-symplectic paths} associated with the shape $\la/\mu$  a family of $N$ non-intersecting lattice paths with starting points $S_i=(\mu_i^\prime-i+1,2 m-\mu_i^\prime+i-1)$ end points $E_j=(\lambda_j^\prime-j+1,2n+2m-\la_j^\prime+j-1)$ for $1 \leq i,j \leq N$. The step set of these paths is $\{(1,0),(0,1)\}$, and additionally each path must stay weakly above the line $y=x-1$. The weight of a vertical step is $1$ and a horizontal step has weight $x_i$ if it is the $(2i)$-th step of a path and weight $x_i^{-1}$ if it is the $(2i-1)$-st step of a path. Equivalently, the weight of a horizontal step starting at $(a,b)$ is $x_{(a+b)/2-m+1}^{-1}$ if $a+b$ is even and $x_{(a+b+1)/2-m}$ if $a+b$ is odd. 

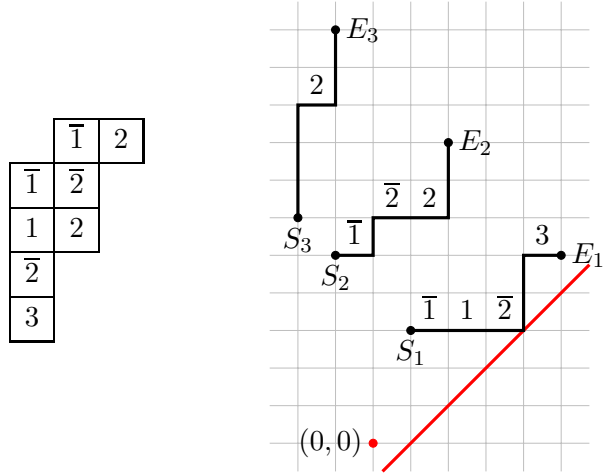
\begin{figure}[htb]
	\centering
	\begin{ytableau}
	\none & \ov{1} &2\\
	\ov{1} & \ov{2}\\
	1 & 2 \\
	\ov{2} \\
	3
	\end{ytableau}
	\qquad\qquad
	\begin{tikzpicture}[scale=.5,baseline=(current bounding box.center)]
	\draw[help lines,step=1cm, lightgray] (-2.75,-.75) grid (5.75,11.75);
	\draw[very thick, red] (.25,-.75) -- (5.75,4.75);
	\filldraw[red] (0,0) circle (3pt) node[left]{\color{black}$(0,0)$};
	\filldraw[black] (1,3) circle (3pt) node[below]{\color{black}$S_1$};
	\filldraw[black] (-1,5) circle (3pt) node[below]{\color{black}$S_2$};
	\filldraw[black] (-2,6) circle (3pt) node[below]{\color{black}$S_3$};
	\filldraw[black] (5,5) circle (3pt) node[right]{\color{black}$E_1$};
	\filldraw[black] (2,8) circle (3pt) node[right]{\color{black}$E_2$};
	\filldraw[black] (-1,11) circle (3pt) node[right]{\color{black}$E_3$};
	\draw[very thick] (1,3) --node[above]{$\ov 1$} (2,3)  --node[above]{$1$} (3,3) --node[above]{$\ov 2$} (4,3) -- (4,5) --node[above]{$3$} (5,5);
	\draw[very thick] (-1,5) --node[above]{$\ov 1$} (0,5) --(0,6) -- node[above]{$\ov 2$} (1,6)-- node[above]{$2$} (2,6) -- (2,8);
	\draw[very thick] (-2,6) -- (-2,9) --node[above]{$2$} (-1,9)-- (-1,11);
	\end{tikzpicture}
	\caption{A skew $(3,2)$-symplectic tableau of shape $(3,2,2,1,1)/(1)$ (left) and its associated family of $(3,2)$-symplectic paths (right).}
	\label{fig:symplectic tableau}
\end{figure}

There is a  bijection between skew $(n,m)$-symplectic tableaux and families of $(n,m)$-symplectic paths that is very similar to the one in the Schur case. Namely, given a skew $(n,m)$-symplectic  tableau $T$, the path starting at $S_i$ of the corresponding family of symplectic paths is obtained by associating the $j$-th step of the path with the $j$-th element in the sequence 
$\ov{1},1,\ov{2},2,\ldots$ and letting a step being horizontal if and only if the corresponding element in the sequence appears in the $i$-th column. 
An example is given in Figure~\ref{fig:symplectic tableau}.

\subsection{Skew \texorpdfstring{$(n,m)$}{(n,m)}-odd orthogonal semistandard Young tableaux}
\label{sec:soYT}
Let $\la,\mu$ be partitions such that $\mu\subseteq\la$, $l(\mu)\leq m$ and
$l(\la)\leq n+m$.
A \defn{skew $(n,m)$-odd orthogonal semistandard Young tableau} (or \defn{skew $(n,m)$-odd orthogonal tableau}) associated with the shape $\la/\mu$ is a skew semistandard Young tableau of shape $\la/\mu$ which is filled by 
\[
\wh{1} < \ov{1} < 1 < \cdots < \wh{n} < \ov{n} < n
\]
and satisfies
\begin{itemize}
\item the \defn{modified $m$-symplectic condition}: the entries in row $m+i$ are at least $\wh{i}$, and 
\item the \defn{$m$-odd orthogonal condition}: the symbol $\wh{i}$ can only appear in the first column of the $(m+i)$-th row.
\end{itemize}
Note that, unlike in \cite{KoiTer90}, we use $\wh{i}$ instead of $\musNatural{}_i$.
Denote by $\SOT_{\la/\mu}^{(n,m)}$ the set of skew $(n,m)$-odd orthogonal  tableaux of shape $\la/\mu$. We define the weight $\x^T$ of a tableau $T \in \SOT_{\la/\mu}^m$ as
\[
\x^T =\prod_{i=1}^n x_i^{(\#\text{ of $i$'s in $T$}) -(\#\text{ of $\ov{i}$'s in $T$})}.
\]
The \defn{skew $m$-odd orthogonal character} $\so_{\la/\mu}^m(\x)$ is the multivariate generating function of skew $(n,m)$-odd orthogonal tableaux of shape $\la/\mu$:
\[
\so_{\la/\mu}^m(\x) = \sum_{T \in \SOT_{\la/\mu}^{(n,m)}}\x^T.
\]

A \defn{family of $(n,m)$-odd orthogonal paths} associated with the shape  $\la/\mu$ is a family of non-intersecting $(n,m)$-symplectic paths for which we extend the step set by a diagonal step $(1,1)$ which can only occur when the starting point is of the form $(i,i)$. The weight of a family of $(n,m)$-odd orthogonal paths is the product of the weights of all steps, where horizontal and vertical steps have the same weight as in the symplectic case and diagonal steps have weight $1$. 

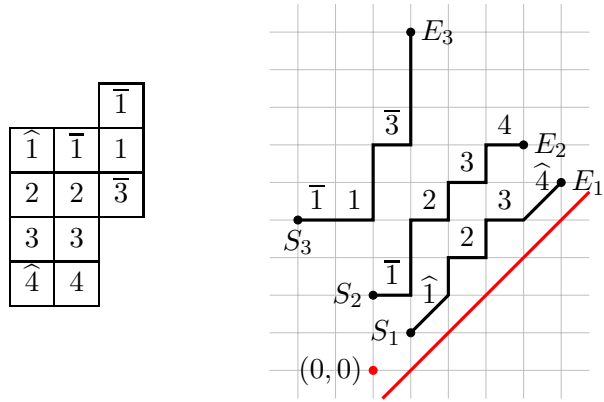
\begin{figure}[htb]
	\centering
				\begin{ytableau}
					\none & \none & \ov 1\\
					\wh{1} & \ov 1 & 1 \\
					2 & 2 & \ov 3\\
					3 & 3 \\
					\wh{4} & 4
				\end{ytableau}
			\qquad\qquad
			\begin{tikzpicture}[baseline=(current bounding box.center),scale=.5]
				\draw[help lines,step=1cm, lightgray] (-2.75,-.75) grid (5.75,9.75);
				\draw[very thick, red] (.25,-.75) -- (5.75,4.75);
				\filldraw[red] (0,0) circle (3pt) node[left]{\color{black}$(0,0)$};
				\filldraw[black] (1,1) circle (3pt)node[left]{\color{black}$S_1$};
				\filldraw[black] (0,2) circle (3pt)node[left]{\color{black}$S_2$};
				\filldraw[black] (-2,4) circle (3pt)node[below]{\color{black}$S_3$};
				\filldraw[black] (5,5) circle (3pt)node[right]{\color{black}$E_1$};
				\filldraw[black] (4,6) circle (3pt)node[right]{\color{black}$E_2$};
				\filldraw[black] (1,9) circle (3pt)node[right]{\color{black}$E_3$};
				\draw[very thick] (1,1) --node[above]{$\wh 1$} (2,2) -- (2,3) --node[above]{$2$} (3,3) -- (3,4) --node[above]{$3$} (4,4)  --node[above]{$\wh 4$} (5,5);
				\draw[very thick] (0,2) --node[above]{$\ov 1$} (1,2) -- (1,4) --node[above]{$2$} (2,4) -- (2,5) --node[above]{$3$} (3,5) -- (3,6) --node[above]{$4$} (4,6);
				\draw[very thick] (-2,4) --node[above]{$\ov 1$} (-1,4) --node[above]{$1$} (0,4) -- (0,6) --node[above]{$\ov 3$} (1,6) -- (1,9);
		\end{tikzpicture}
		\caption{A skew $(4,1)$-odd orthogonal tableau of shape $(3,3,3,2,2)/(2)$ (left) and its associated family of $(4,1)$-odd orthogonal paths (right).}
		\label{fig:so tableau}
\end{figure}

We obtain a weight preserving bijection from skew $(n,m)$-odd orthogonal tableaux of shape $\la/\mu$ to families of $(n,m)$-odd orthogonal paths associated with the shape $\la/\mu$ as follows. We follow the same procedure as in the symplectic setting by interpreting first each $\wh{i}$ entry as $\ov{i}$. Note that each horizontal step corresponding to an $\wh{i}$ entry ends at the line $y=x-1$ and therefore has to be followed by a vertical step. Now replace each of these horizontal steps coming from an $\wh{i}$ entry and its following vertical step by a diagonal step. See Figure~\ref{fig:so tableau} for an example.

\subsection{Skew \texorpdfstring{$(n,m)$}{(n,m)}-even orthogonal  semistandard Young tableaux}
\label{sec:oYT}

As before, let $\mu\subseteq\la$ be partitions such that $l(\mu)\leq m$ and 
$l(\la)\leq n+m$. A \defn{skew $(n,m)$-even orthogonal  semistandard Young tableau} (or \defn{skew $(n,m)$-even orthogonal  tableau}) of shape $\la/\mu$ is a semistandard Young tableau of shape $\la/\mu$ which is filled by the symbols
\[
\wc{1} < \wh{1} < \ov{1} < 1 < \cdots <\wc{n} < \wh{n} < \ov{n} < n
\]
and satisfies  
\begin{itemize}
\item the \defn{modified $m$-symplectic condition}: the entries in row $m+i$ are at least $\wh{i}$,
\item the \defn{modified $m$-odd orthogonal condition}: the symbol $\wh{i}$ can only appear in the first column of the $(m+i)$-th row and it appears if and only if the entry above is $\wc{i}$, and 
\item the \defn{$m$-even orthogonal condition}: if $\ov{i}$ appears in the first column of the $(m+i)$-th row and $i$ also appears in the same row, then there is an $\ov{i}$ immediately above this $i$ entry.
\end{itemize}
Note that, unlike in \cite{KoiTer90}, we use the symbols $\wc{i}$ and $\wh{i}$ instead of $\musSharp{}_i$  $\musFlat{}_i$; this corresponds to interchanging the roles of $\wc{i}$ and $\wh{i}$ if compared to \cite{AyyerFischer20}. 

We denote the set of skew $(n,m)$-even orthogonal  tableaux of shape $\la/\mu$ by $\OT_{\la/\mu}^{(n,m)}$. The weight $\x^T$ of a skew $(n,m)$-even orthogonal tableau $T$ is defined as
\[
\x^T =\prod_{i=1}^n x_i^{(\#\text{ of } i \text{ entries in }T) -(\#\text{ of } \ov{i} \text{ entries in }T)}.
\]
The \defn{skew $m$-even orthogonal character} $\o_{\la/\mu}^m(\x)$ is the multivariate generating function of skew $(n,m)$-even orthogonal tableaux of shape $\la/\mu$, so
\[
\o_{\la/\mu}^m (\x)= \sum_{T \in \OT_{\la/\mu}^{(n,m)}}\x^T.
\]

We say that a family of lattice paths is \defn{strongly non-intersecting} if there is no intersection between any pair of paths when considering them 
as subsets of $\mathbb{R}^2$. (For our application of the Lindstr\"om--Gessel--Viennot Lemma~\ref{lem:LGV}, we will also need the 
notion of \emph{weakly non-intersecting} in Section~\ref{sec:LGV}, where we only forbid intersections at lattice points that are endpoints of steps in both paths.)

A \defn{family of $(n,m)$-even orthogonal paths} associated with the shape $\la/\mu$ is a family of strongly non-intersecting symplectic paths for which we extend the step set by a horizontal step $(2,0)$, called an \defn{$\o$-horizontal step}, which can only occur starting at a point of the form $(i-2,i)$, and where the family of paths does not have any trapped position as defined below. 
We draw $\o$-horizontal steps as arcs to avoid confusion.
\begin{figure}[htb]
\centering
\begin{tikzpicture}[scale=.5]
	\draw[help lines,step=1cm, lightgray] (-2.75,.25) grid (2.75,5.75);
	\draw[very thick, red] (1.25,.25) -- (2.75,1.75);
	\draw[very thick] (1,1) -- (2,1) -- (2,2);
	\draw[very thick] (0,2) -- (1,2) -- (1,3);
	\draw[very thick] (-1,3) -- (0,3) -- (0,4);
	\draw[very thick] (-2,4) -- (-2,5) -- (-1,5);
	\filldraw[red] (-1,4) circle (3pt);
\end{tikzpicture}
\caption{The local configuration around a trapped position which is marked as a red dot.}
\label{fig:trapped position}
\end{figure}
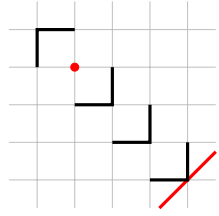

For positive integers $i,d$, we call the position $(m+i-d,m+i+d-1)$ \defn{trapped} if the following is satisfied:
\begin{itemize} 
\item  the lattice point $(m+i-d,m+i+d-1)$ is not contained in any path,
\item for each $d' \in \{0,\ldots,d-1\}$, there is a lattice path which passes through $(m+i-d',m+i+d'-1)$ by a horizontal step followed by a vertical step,
and
\item there is a path passing through the point $(m+i-d-1,m+i+d)$ by a vertical step followed by a horizontal step. 
\end{itemize} 
See Figure~\ref{fig:trapped position} for an example with $d=3$.

We now describe the weight preserving bijection from $(n,m)$-even orthogonal tableaux 
of shape $\la/\mu$ to families of $(n,m)$-even orthogonal paths associated with the same shape that are strongly non-intersecting. First we interpret each $\wc{i}$ as $\ov i$ and each $\wh{i}$ as $i$ and apply the map from $(n,m)$-symplectic tableaux to families of $(n,m)$-symplectic paths. Then we replace the pairs of horizontal steps associated with $\wc{i},\wh{i}$ coming from the 
modified $m$-odd orthogonal condition by an $\o$-horizontal step. For an example see Figure~\ref{fig:o tableau}.
In order to see that this map is a bijection it suffices to check that a tableau contradicts the $m$-even orthogonal condition if and only if the corresponding family of paths has a trapped position, which is done next.

\begin{figure}[htb]
	\centering
	\begin{ytableau}
		\none & \none & \ov 1 & 1\\
		\ov 1 & \ov 1 & 1 & \ov 2\\
		\wc{3} & \ov 3 &  3 & 4\\
		\wh{3} & \ov 4 \\
		\ov 4 & 4
	\end{ytableau}
	\qquad\qquad
	\begin{tikzpicture}[scale=.5,baseline=(current bounding box.center)]
		\draw[help lines,step=1cm, lightgray] (-3.75,-.75) grid (5.75,10.75);
		\draw[very thick, red] (.25,-.75) -- (5.75,4.75);
		\filldraw[red] (0,0) circle (3pt) node[left]{\color{black}$(0,0)$};
		\filldraw[black] (1,1) circle (3pt)node[left]{\color{black}$S_1$};
		\filldraw[black] (0,2) circle (3pt)node[left]{\color{black}$S_2$};
		\filldraw[black] (-2,4) circle (3pt)node[below]{\color{black}$S_3$};
		\filldraw[black] (-3,5) circle (3pt)node[below]{\color{black}$S_4$};
		\filldraw[black] (5,5) circle (3pt)node[right]{\color{black}$E_1$};
		\filldraw[black] (4,6) circle (3pt)node[right]{\color{black}$E_2$};
		\filldraw[black] (1,9) circle (3pt)node[right]{\color{black}$E_3$};
		\filldraw[black] (0,10) circle (3pt)node[right]{\color{black}$E_4$};
		\draw[very thick] (1,1) --node[above]{$\ov 1$} (2,1) -- (2,4) to[out=45, in=135] (4,4)  --node[above]{$\ov 4$} (5,4) -- (5,5);
		\node at (2.5,5) {$\wc{3}$};
		\node at (3.5,5) {$\wh{3}$};
		\draw[very thick] (0,2) --node[above]{$\ov 1$} (1,2) -- (1,5) --node[above]{$\ov 3$} (2,5) -- (2,6) --node[above]{$\ov 4$}(3,6) --node[above]{$4$} (4,6);
		\draw[very thick] (-2,4) --node[above]{$\ov 1$} (-1,4) --node[above]{$1$} (0,4) -- (0,7) --node[above]{$3$} (1,7) -- (1,9);
		\draw[very thick] (-3,5) -- (-3,6) --node[above]{$1$} (-2,6) -- node[above]{$\ov 2$}(-1,6)--(-1,10)--node[above]{$4$} (0,10);
	\end{tikzpicture}
	\caption{A skew $(4,1)$-even orthogonal tableau of shape $(4,4,4,2,2)/(2)$ (left) and its associated family of $(4,1)$-even orthogonal paths (right).}
	\label{fig:o tableau} 
\end{figure}
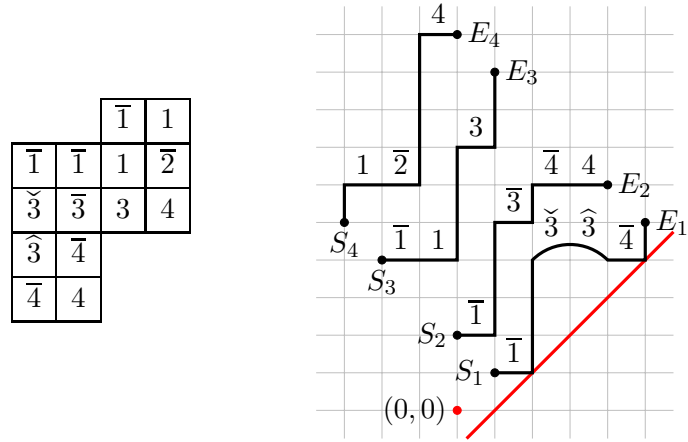

Assume that $T$ is a tableau contradicting the $m$-even orthogonal-condition, i.e., there exists an integer $i$ such that the $(m+i)$-th row starts with an entry $\ov{i}$ and also contains an entry $i$ whose entry above is not $\ov{i}$. It is immediate that the top neighbour of the first $i$ in this row can not be $\ov{i}$ since this would contradict the condition that rows are weakly increasing. Denote by $d$ the number of $\ov{i}$ entries in the $(m+i)$-th row. When looking at the corresponding family of paths, the $(d'+1)$-st entry $\ov{i}$ corresponds to a horizontal step ending at the point $(m+i-d',m+i+d'-1)$. Since the bottommost of these paths touches the line $y=x-1$ at $(m+i,m+i-1)$ and all paths are strongly non-intersecting, each of these horizontal steps are followed by a vertical step. The first entry $i$ in the $(m+i)$-th row corresponds to a horizontal step starting at $(m+i-d-1,m+i+d)$. Since the top neighbour of this $i$ is not $\ov{i}$, the step before has to be a vertical step. This implies that the position $(m+i-d,m+i+d-1)$ is trapped.

\section{Combinatorial proofs of the dual Jacobi--Trudi formulae}
\label{sec:dualJacobiTrudi}
The proofs of the dual Jacobi--Trudi formulae for all of the skew characters under consideration --- skew symplectic, skew odd orthogonal and skew even orthogonal --- follow a similar scheme: we interpret the respective tableaux columnwise as non-intersecting lattice paths as seen in Section~\ref{sec:tableaux}.
Thus, each column corresponds to a lattice path whose generating function can be written in terms of elementary symmetric functions in the alphabet $\x^\pm=(x_1^{-1},x_1,\dots,x_n^{-1},x_n)$. Applying the Lindstr\"om--Gessel--Viennot Lemma~\ref{lem:LGV} then yields a determinantal formula.

The proofs increase in complexity and a brief summary is as follows.

\begin{itemize} 
\item For the \defn{skew symplectic case}, we have to compute the generating function of lattice paths consisting of unit horizontal and vertical steps in the positive direction which do not cross the line $y=x-1$. This is achieved by using a modified reflection principle  (Lemma~\ref{lem:modifiedreflection}) by Fulmek and Krattenthaler \cite{FulKra97}, which provides a difference of two elementary symmetric functions as the generating function, see  Lemma~\ref{lem:SymplecticPaths}.
\item In the case of \defn{skew odd orthogonal characters}, we need to compute the generating function of  lattice paths which may additionally have diagonal steps along the line $y=x$. For each fixed number of such diagonal steps in the lattice path, we obtain a difference of elementary symmetric functions in Lemma~\ref{lem:kSpecialOrthogonalPaths}. Adding these differences together allow telescopic cancelling, which finally yields a sum of two elementary symmetric functions in Corollary~\ref{cor:SpecialOrthogonalPaths}.
\item In the case of \defn{skew even orthogonal characters}, we allow horizontal double steps ending on the line $y=x$ instead of diagonal steps. By similar means (Lemma~\ref{lem:kOrthogonalPaths}), we also obtain a sum of two elementary symmetric functions as the generating function for these lattice paths in Corollary~\ref{cor:OrthogonalPaths}. However, applying the Lindstr\"om--Gessel--Viennot Lemma~\ref{lem:LGV} in this case also results in families of lattice paths that do not correspond to a skew $(n,m)$-even orthogonal tableau. We provide a sign-reversing involution between those and families of lattice paths with trapped positions at the end of the section.
\end{itemize} 

\subsection{Modified reflection principle}

One of the main tools in proving the dual Jacobi--Trudi formulae is a \defn{modified reflection principle} which we present next. 

A lattice point $(x,y) \in \mathbb{Z}^2$ is said to be \defn{even} if $x+y$ is even, otherwise it is said to be \defn{odd}. Suppose we have a lattice path starting in $P=(a,b)$ that consists of unit horizontal and vertical steps in the positive direction. In addition, we assume that $P$ is an {even point}. 
We assign weights to the steps as follows. Vertical steps have weight~$1$, whereas the weights of horizontal steps are given by a modified $e$-labelling: the step from $(i,j)$ to $(i+1,j)$ has weight $x_{(i+j-a-b)/2+1}^{-1}$ if $(i,j)$ is even and $x_{(i+j-a-b+1)/2}$ if $(i,j)$ is odd. 
Now the weight of the path is the product of the weights of its steps. 

Furthermore, consider a line $y=x+d$ such that $d$ is even and $P$ lies above that diagonal line, that is, $b>a+d$. The modified reflection principle will 
enable us to derive combinatorially a formula for the generating function of such lattice paths that start at $P$ and that have no intersection with the line $y=x+d$. This is done by computing the generating function of those paths that have an intersection with $y=x+d$ and then subtracting it from the generating function of all lattice paths.

\begin{lem} \label{lem:modifiedreflection}
Let $a,b,c,d,f \in \mathbb{Z}$ such that $a+b$ and $d$ are even, $b>a+d$ and
$f>c+d$. There is a weight-preserving bijection between lattice paths with unit horizontal and vertical steps that start at the point $P=(a,b)$ and end at $(c,f)$, and that have an intersection with the line $y=x+d$, and lattice paths with unit horizontal and vertical steps that start at the reflected point $P'=(b-d,a+d)$ of $P$ along $y=x+d$ and also end at $(c,f)$.
\end{lem} 

\begin{proof} 
The proof is illustrated in Figure~\ref{fig:modified reflection principle}.
\begin{figure}[htb] 
	\centering
\begin{tikzpicture}[scale=.6]
	\draw[help lines,step=1cm, lightgray] (.5,.5) grid (12.5,12.5);
	\draw[very thick, red] (.5,.5) -- (12.5,12.5);
	\draw[very thick](1,3) -- (1,4) --node[above]{$x_1$} (2,4) -- (2,6) --node[above]{$\ov x_3$} (3,6) -- (3,7) --node[above]{$\ov x_4$} (4,7) --node[above]{$x_4$} (5,7) --node[above]{$\ov x_5$} (6,7) -- (6,8) --node[above]{$\ov x_6$} (7,8) --node[above]{$x_6$} (8,8) --node[above]{$\ov x_7$} (9,8) --node[above]{$ x_7$} (10,8) -- (10,9) --node[above]{$x_8$} (11,9) -- (11,10) -- (11,11) -- (11,12.5);
	\draw[very thick](3,1) -- (3,2) --node[above]{$x_1$} (4,2) --node[above]{$\ov x_2$} (5,2) --node[above]{$x_2$} (6,2) --node[above]{$\ov x_3$} (7,2) -- (7,3) -- (7,4) -- (7,5) --node[above]{$\ov x_5$} (8,5) -- (8,6) -- (8,7) -- (8,8);
	
	\draw[dashed,gray, thick] (1,3) -- (3,1);
	\draw[dashed,gray, thick] (2,4) -- (4,2);
	\draw[dashed,gray, thick] (2,6) -- (6,2);
	\draw[dashed,gray, thick] (3,7) -- (7,3);
	\draw[dashed,gray, thick] (5,7) -- (7,5);
	\draw[dashed,gray, thick] (6,8) -- (8,6);
	
	\filldraw[cyan] (1,3) circle (3.5pt) node[below]{\color{black} $P$};
	\filldraw[cyan] (2,4) circle (3.5pt);
	\filldraw[cyan] (2,6) circle (3.5pt);
	\filldraw[cyan] (3,7) circle (3.5pt);
	\filldraw[cyan] (5,7) circle (3.5pt);
	\filldraw[cyan] (6,8) circle (3.5pt);
	\filldraw[cyan] (8,8) circle (3.5pt)node[below right]{\color{black} $Q$};;
	\filldraw[cyan] (3,1) circle (3.5pt)node[below]{\color{black} $P'$};
	\filldraw[cyan] (4,2) circle (3.5pt);
	\filldraw[cyan] (6,2) circle (3.5pt);
	\filldraw[cyan] (7,3) circle (3.5pt);
	\filldraw[cyan] (7,5) circle (3.5pt);
	\filldraw[cyan] (8,6) circle (3.5pt);
	\filldraw[red] (1,4) \Square{3.5pt};
	\filldraw[red] (2,5) \Square{3.5pt};
	\filldraw[red] (3,6) \Square{3.5pt};
	\filldraw[red] (4,7) \Square{3.5pt};
	\filldraw[red] (6,7) \Square{3.5pt};
	\filldraw[red] (7,8) \Square{3.5pt};
	\filldraw[red] (3,2) \Square{3.5pt};
	\filldraw[red] (5,2) \Square{3.5pt};
	\filldraw[red] (7,2) \Square{3.5pt};
	\filldraw[red] (7,4) \Square{3.5pt};
	\filldraw[red] (8,5) \Square{3.5pt};
	\filldraw[red] (8,7) \Square{3.5pt};
\end{tikzpicture}
\caption{The modified reflection principle where $\ov x_i\coloneqq x_i^{-1}$.}
\label{fig:modified reflection principle}
\end{figure}
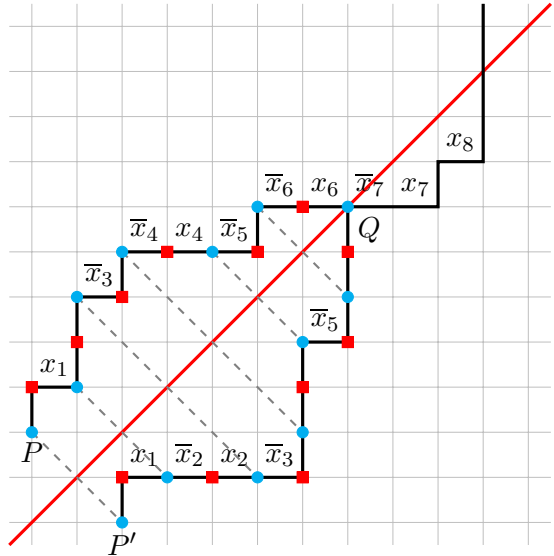
Suppose our path touches the line for the first time at the point~$Q$, when traversing it starting at $P$. 
The modified reflection of the path from $P$ to $Q$ along the line $y=x+d$ works as follows: $P$ is reflected in the usual way, so it is mapped to $P'=(b-d,a+d)$. The same applies to all other even points and to odd points at which the path does not turn; these are odd points that lie in between either two vertical or two horizontal steps. The remaining points are odd points in which the path turns. We reflect these points in such a way that the directions of the turns are maintained. Concretely, if an odd point $(x,y)$ comes with a left turn (a horizontal step followed by a vertical step), it is mapped to $(y-d+1,x+d-1)$; if it comes with a right turn (a vertical step followed by a horizontal step), it is mapped to $(y-d-1,x+d+1)$. 

Note that the modification of the usual reflection ensures that the mapping is weight-preserving. 
\end{proof} 

\subsection{The Lindstr\"om--Gessel--Viennot Lemma}
\label{sec:LGV}

Another important tool is the well-known Lindstr\"om--Gessel--Viennot Lemma \cite{GesVie85,GesVie89,Lin73}. We state it in a form that is convenient for us. We need the following two notions.

\begin{itemize}
\item We say that a family of lattice paths is \defn{weakly non-intersecting} if no pair of paths has an intersection at a lattice 
point that is the endpoint of steps in both paths.
\item We say that a family of lattice paths is \defn{strongly non-intersecting} if no pair of paths  intersects when considering them as 
subsets of $\mathbb{R}^2$.
\end{itemize}

One may interpret the first notion of non-intersecting when considering the paths as graphs, and the second notion when considering their natural embeddings in $\mathbb{R}^2$. The background for these notions is that the Gessel--Viennot sign-reversing involution can only be applied if 
there is a pair of paths that intersect at a lattice point that is the endpoint of steps in both paths,  which is precisely the case when 
the family is \emph{not} weakly non-intersecting. 
Also note that the two notions of non-intersecting are equivalent for families of Schur paths, of $m$-symplectic paths, and of $m$-odd orthogonal paths, however, they differ for families of $m$-even orthogonal paths due to the steps of length $2$. Thus we may simply say non-intersecting in the first two cases.

\begin{lem}[The Lindstr\"om--Gessel--Viennot Lemma]
\label{lem:LGV} 
We consider lattice paths on the lattice $\mathbb{Z}^2$ which is equipped with edge weights and 
where the allowed steps are given by a finite subset of $\mathbb{Z}^2$ such that the step set does not allow 
self-intersections of paths. Let $S_1,\ldots,S_n$ and $E_1,\ldots,E_n$ be lattice points and let ${\mathcal P}(S_i \to E_j)$ denote the generating function of lattice paths from $S_i$ and $E_j$ with respect to the step set and the edge weights. Then 
\[
\det_{1 \leq i,j \leq n}( {\mathcal P}(S_i \to E_j) )
\]
is the signed generating function of 
families of $n$ lattice paths  such that the following is satisfied.
\begin{itemize} 
\item The  $S_1,S_2,\ldots,\allowbreak S_n$ are the starting points and 
the $E_1,E_2,\ldots,E_n$ are the ending points. 
\item The paths are weakly non-intersecting. 
\item The sign 
is $\sgn \sigma$ where $\sigma$ is the permutation of $\{1,2,\ldots,n\}$ such that $S_i$ is connected to $E_{\sigma(i)}$ for $i=1,2,\ldots,n$.
\end{itemize}
\end{lem} 

\subsection{Skew symplectic characters}

For the skew symplectic character $\sp^m_{\lambda/\mu}$, we have to compute the generating function of families of non-intersecting lattice paths from $(\mu'_j-j+1,2m-\mu'_j+j-1)$ to $(\lambda'_i-i+1,2n+2m-\lambda_i'+i-1)$ for $1 \leq i,j \leq N$ with $\la_1\leq N$, $l(\mu)\leq m$ and $l(\la)\leq n+m$ that consist of unit horizontal and vertical steps in the positive direction and that do not cross the line $y=x-1$, see Section~\ref{sec:spYT}. In order to apply Lemma~\ref{lem:LGV}, we first compute the generating function of paths between pairs of starting and
ending points.

\begin{lem}\label{lem:SymplecticPaths}
Let $a,b,c\in\mathbb{Z}$ and $n\in\mathbb{N}$ such that $a+b$ is even 
and the points $(a,b)$ and $(c,2n+a+b-c)$ lie strictly above the line $y=x-1$.
Then, with the path set-up of Lemma~\ref{lem:modifiedreflection},
the generating function of paths from $(a,b)$ to 
$(c,2n+a+b-c)$ is given by
\begin{equation*}
e_{c-a}(\x^\pm) - e_{c-b-2}(\x^\pm),
\end{equation*}
where $\x^\pm=(x_1,x_1^{-1},\dots,x_n,x_n^{-1})$.
\end{lem}

\begin{proof}
	We use the modified reflection principle from Lemma~\ref{lem:modifiedreflection} with $(a,b,c,d,f)\mapsto(a,b,c,-2,2n+a+b-c)$ for this proof. The generating function of all lattice paths from $(a,b)$ to $(c,2n+a+b-c)$ with the given step set regardless of whether they cross the line $y=x-1$ is $e_{c-a}(\x^\pm)$. 
	
	Next, we derive the generating function of all lattice paths from $(a,b)$ to $(c,2n+a+b-c)$ that do cross that line. Those paths touch at least once the line $y=x-2$. We take the initial part of a path from $(a,b)$ to that first intersection point and reflect it in a weight-preserving way according to the modified reflection principle. This yields paths from $(b+2,a-2)$ to $(c,2n+a+b-c)$. The generating function of these paths is $e_{c-b-2}(\x^\pm)$, which we need to subtract from $e_{c-a}(\x^\pm)$ to obtain the result.
\end{proof}

In view of the above, the generating function of lattice paths from $(\mu'_j-j+1,2m-\mu'_j+j-1)$ to $(\lambda'_i-i+1,2n+2m-\la_i'+i-1)$  which do not cross the line $y=x-1$ is given by
\begin{equation*}
	e_{\lambda'_i -\mu'_j-i+j}(\x^\pm) - e_{\lambda'_i + \mu'_j-i-j-2m}(\x^\pm).
\end{equation*}
Note the requirement that $(\la'_i-i+1,2n+2m-\la_i'+i-1)$ lies above
the line $y=x-1$ is equivalent to $l(\la)\leq n+m$.
Applying the Lindstr\"om--Gessel--Viennot Lemma~\ref{lem:LGV} finally yields
\begin{equation*}
	\sp^m_{\lambda/\mu}(\x) = \det_{1 \leq i,j \leq N} \big( e_{\lambda'_i -\mu'_j-i+j}(\x^\pm) - e_{\lambda'_i + \mu'_j-i-j-2m}(\x^\pm) \big), 
\end{equation*}
and this concludes the combinatorial proof of Theorem~\ref{thm:dual} \eqref{dual_sp}.

\subsection{Skew odd orthogonal characters}

For the skew odd orthogonal character $\so_{\lambda/\mu}^m$, we have to compute as before the generating function of families of non-intersecting lattice paths from $(\mu'_j-j+1,2m-\mu'_j+j-1)$ to $(\lambda'_i-i+1,2n+2m-\la_i'+i-1)$ for $1 \leq i,j \leq N$ with $\la_1\leq N$, $l(\mu)\leq m$ and $l(\la)\leq n+m$ that do not cross the line $y=x-1$; see
Section~\ref{sec:soYT}. In this case, however, we have a bigger step set. We allow unit horizontal and vertical steps in the positive direction and, in addition, diagonal steps $(1,1)$ on the line $y=x$ which have weight~$1$. In order to compute this generating function, we first derive the generating function of individual paths with a given number~$k$ of diagonal steps and then we sum over all non-negative integers~$k$.

\begin{lem}\label{lem:kSpecialOrthogonalPaths}
Let $k$ be a positive integer and let $\{(1,0),(0,1),(1,1)\}$ be the step set such that diagonal steps are only allowed along the line $y=x$. 
Also let $a,b,c\in\mathbb{Z}$ and $n\in\mathbb{N}$ such that $a+b$ is even and
the points $(a,b)$ and $(c,2n+a+b-c)$ lie strictly above the line $y=x-1$.
Then the generating function of lattice paths from $(a,b)$ to $(c,2n+a+b-c)$ that do not cross the line $y=x-1$ and that have exactly $k$ diagonal steps is
	\begin{equation*}
		e_{c-b-k}(\x^\pm) - e_{c-b-k-2}(\x^\pm),
	\end{equation*}
where $\x^\pm=(x_1,x_1^{-1},\dots,x_n,x_n^{-1})$.
\end{lem}

\begin{proof}
	First, we map a path with $k$ diagonal steps to a path which only consists of unit horizontal and vertical steps as follows. Replace every diagonal step by two vertical steps. In the process, we keep the terminal part of the path and shift the initial part accordingly. This is illustrated in Figure~\ref{fig:lemmasp}.
	We obtain a lattice path from $(a+k,b-k)$ to $(c,2n+a+b-c)$ that does not cross the line $y=x-2k-1$ but intersects the line $y=x-2k$. 
	
	\begin{figure}
\begin{tikzpicture}
\begin{scope}[scale=.5]
	\draw[help lines,step=1cm, lightgray] (-.5,-1.5) grid (11.5,11.5);
	\draw[very thick, red] (-.5,-.5) -- (11.5,11.5);
	\draw[very thick] (0,2) --++ (1,0) --++ (0,2) --++(2,0) --++ (2,2) --++ (1,0) --++ (0,2) --++ (1,0) --++ (0,1) --++ (1,0) --++ (1,1) --++(0,1.5);
	\filldraw[black] (0,2) circle (3pt) node[below]{$P$};	
	\node at (11.5,12) {$y=x-1$};
	\draw[very thick,cyan] (3,4) --++(2,2);
	\draw[very thick,cyan] (8,9) --++(1,1);
\end{scope}
\end{tikzpicture}
\begin{tikzpicture}
\begin{scope}[scale=.5]
	\draw[help lines,step=1cm, lightgray] (-.5,-1.5) grid (11.5,11.5);
	\draw[very thick, red] (-.5,-.5)-- (11.5,11.5);
	\draw[very thick, red, dashed] (3.5,-1.5)-- (11.5,6.5);
	\draw[very thick] (3,-1) --++ (1,0) --++ (0,2) --++(2,0) --++ (0,4) --++ (1,0) --++ (0,2) --++ (1,0) --++ (0,1) --++ (1,0) --++ (0,2) --++(0,1.5);
	\filldraw[black] (3,-1) circle (3pt) node[left]{$P'$};
	\node at (11.5,12) {$y=x-1$};
	\node at (11.5,7) {$y=x-2k$};
	\draw[very thick,cyan] (6,1) --++(0,4);
	\draw[very thick,cyan] (9,8) --++(0,2);
\end{scope}
\end{tikzpicture}
\caption{Situation in Lemma~\ref{lem:kSpecialOrthogonalPaths} when replacing diagonal steps by vertical steps of length $2$.}
\label{fig:lemmasp}
\end{figure}
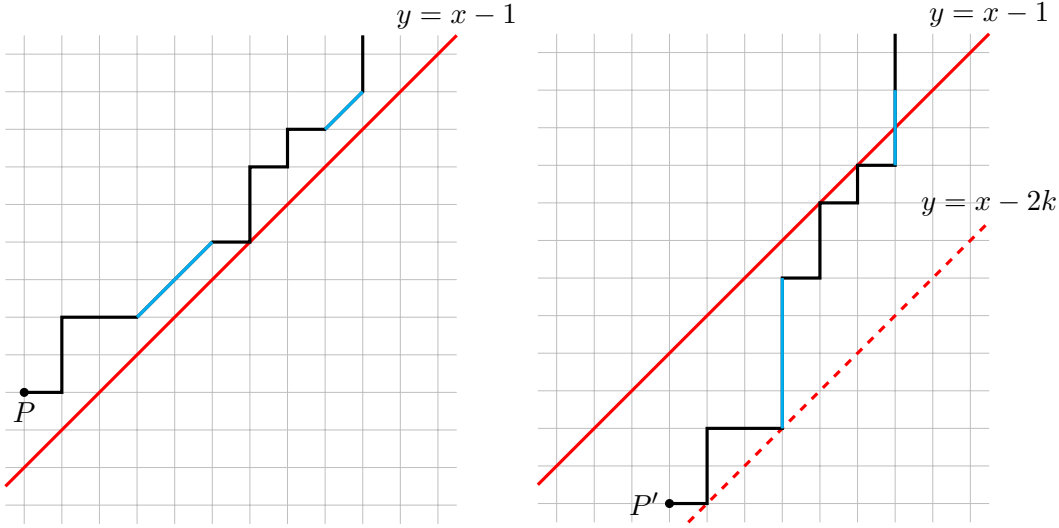	
	
	We observe that this mapping is a weight-preserving bijection. For the inverse mapping, we consider a lattice path from $(a+k,b-k)$ to $(c,2n+a+b-c)$ not crossing the line $y=x-2k-1$ but intersects the line $y=x-2k$. Since the path goes to $(c,2n+a+b-c)$, there must be a point where the path touches the line $y=x-2k$ and continues with at least two vertical steps. We take the rightmost of such points, replace the first two vertical steps by a diagonal step and shift the initial part of the path by $(-1,1)$. We thus obtain a path from $(a+k-1,b-k+1)$ to $(c,2n+a+b-c)$ not crossing the line $y=x-2k+1$ but intersecting the line $y=x-2k+2$. As before, there is a point where the path intersects with the line $y=x-2k+2$ and continues with at least two vertical step. We iteratively repeat the step of replacing two vertical steps by a diagonal step at the rightmost occurrence and shifting the initial part of the path by $(-1,1)$ until we have got a path that starts at $(a,b)$ and has $k$ diagonal steps.
	
	With this weight-preserving bijection in mind, the statement in the theorem follows by applying twice the modified reflection principle from Lemma~\ref{lem:modifiedreflection} as follows. 

	To compute the generating function of paths from $(a+k,b-k)$ to $(c,2n+a+b-c)$ that intersect the line $y=x-2k$, we reflect $(a+k,b-k)$ along $y=x-2k$, which yields $(b+k,a-k)$. Hence, the generating function is $e_{c-b-k}(\x^\pm)$. 
	
From the latter generating function, we have to subtract the generating function of paths that cross the line $y=x-2k-1$; note that such paths always have to intersect the line $y=x-2k$ because 
of the end point $(c,2n+a+b-c)$. Paths that cross the line $y=x-2k-1$ definitely intersect the line $y=x-2k-2$. Reflecting $(a+k,b-k)$ along $y=x-2k-2$ gives $(b+k+2,a-k-2)$, which implies that the generating function of paths crossing the line $y=x-2k-1$ is $e_{c-b-k-2}(\x^\pm)$. This concludes the proof of the lemma.
\end{proof}

The generating function of paths from $(a,b)$ to $(c,2n+a+b-c)$ without diagonal steps is $e_{c-a}(\x^\pm)-e_{c-b-2}(\x^\pm)$, which follows from Lemma~\ref{lem:SymplecticPaths}. Next we sum the generating functions for paths with exactly $k$ diagonal steps for all non-negative $k$. This sum turns out to be a telescoping sum which reduces to a sum of two terms after cancellation:
\begin{equation*}
	\big( e_{c-a}(\x^\pm)-e_{c-b-2}(\x^\pm)\big) + \sum_{k \geq 1} \big(e_{c-b-k}(\x^\pm)-e_{c-b-k-2}(\x^\pm) \big) = e_{c-a}(\x^\pm)+e_{c-b-1}(\x^\pm).
\end{equation*}
This leaves us with the following.

\begin{cor}\label{cor:SpecialOrthogonalPaths}
Let $a,b,c\in\mathbb{Z}$ and $n\in\mathbb{N}$ such that $a+b$ is even
and the points $(a,b)$ and $(c,2n+a+b-c)$ lie strictly above the line $y=x-1$.
Then the generating function of lattice paths from $(a,b)$ to $(c,2n+a+b-c)$ 
with step set $\{(1,0),(0,1),(1,1)\}$ that do not cross the line $y=x-1$ and where diagonal steps are only allowed on the line $y=x$ is
	\begin{equation*}
		e_{c-a}(\x^\pm)+e_{c-b-1}(\x^\pm),
	\end{equation*}
where $\x^{\pm}=(x_1,x_1^{-1},\dots,x_n,x_n^{-1})$.
\end{cor}

By setting $a=\mu'_j-j+1$, $b=2m-\mu'_j+j-1$ and $c=\lambda'_i-i+1$ in Corollary~\ref{cor:SpecialOrthogonalPaths} and applying the Lindstr\"om--Gessel--Viennot Lemma (Lemma~\ref{lem:LGV}), we finally obtain
\begin{equation*}
	\so^m_{\lambda/\mu}(\x) = \det_{1 \leq i,j \leq N} \big( e_{\lambda'_i -\mu'_j-i+j}(\x^\pm) + e_{\lambda'_i + \mu'_j-i-j-2m+1}(\x^\pm) \big),
\end{equation*}
and this concludes the combinatorial proof of Theorem~\ref{thm:dual} \eqref{dual_so}.

\subsection{Skew even orthogonal characters}
\label{sec:dualJacobiTrudi_o}

Recall that in the case of skew even orthogonal characters, we have to compute the generating function of strongly
non-intersecting lattice paths from $(\mu'_j-j+1,2m-\mu'_j+j-1)$ to $(\lambda'_i-i+1,2n+2m-\la_i'+i-1)$ for $1 \leq i,j \leq N$ with $\la_1\leq N$, $l(\mu)\leq m$ and $l(\la)\leq n+m$ that do not cross the line $y=x-1$ and where a certain pattern that corresponds to the $m$-even orthogonal condition is not allowed; see Section~\ref{sec:oYT}. The step set is $\{(1,0),(0,1),(2,0)\}$, where the horizontal steps $(2,0)$, called $\o$-horizontal steps, are only allowed to start at points on the line $y=x+2$ (and thus end on the line $y=x$). The $\o$-horizontal steps are equipped with the weight~$1$ and the unit steps have weights according to the standard $e$-labelling. 

We compute the generating function of such paths in Corollary~\ref{cor:OrthogonalPaths} in a similar manner to the previous section. However, we also need to deal with the avoidance of the patterns that take care of the 
$m$-even orthogonal condition, that is, with trapped positions. It turns out that this combines nicely with the application of the 
Lindstr\"om--Gessel--Viennot Lemma~\ref{lem:LGV}. Namely, when applying the lemma to the generating functions of paths from Corollary~\ref{cor:OrthogonalPaths}, we may have pairs of paths that are weakly non-intersecting but not strongly non-intersecting. 
 With our step set such intersections may only occur when $\o$-horizontal steps intersect with two unit vertical steps. Therefore, we present a sign-reversing involution that shows that families of lattice paths with such intersections and with trapped positions indeed cancel after applying the Lindstr\"om--Gessel--Viennot Lemma~\ref{lem:LGV}.

\begin{lem} \label{lem:kOrthogonalPaths}
	Let $k$ be a positive integer and let $\{(1,0),(0,1),(2,0)\}$ be the step set such that $\o$-horizontal steps are only allowed to end on the line $y=x$.
Also let $a,b,c\in\mathbb{Z}$ and $n\in\mathbb{N}$ such that $a+b$ is even and
$(a,b)$ and $(c,2n+a+b-c)$ lie strictly above the line $y=x-1$.
Then the generating function of lattice paths from $(a,b)$ to $(c,2n+a+b-c)$ that do not cross the line $y=x-1$ and that have exactly $k$ $\o$-horizontal steps is
	\begin{equation*}
		\begin{cases}
		e_{c-b-2k+2}(\x^\pm) - e_{c-b-2k-2}(\x^\pm) & \text{if $b-a \geq 2$,}\\
		e_{c-a-2k}(\x^\pm) - e_{c-b-2k-2}(\x^\pm) & \text{otherwise},
		\end{cases}
	\end{equation*}
where $\x^\pm=(x_1,x_1^{-1},\dots,x_n,x_n^{-1})$.
\end{lem}

\begin{proof}
	We proceed in a way similar to the proof of Lemma~\ref{lem:kSpecialOrthogonalPaths}. First, we replace every $\o$-horizontal step by a double vertical step whilst keeping the terminal part of the path but shifting the initial part accordingly. See Figure~\ref{fig:lemmaso} for an illustration. This results in a path from $(a+2k,b-2k)$ to $(c,2n+a+b-c)$ that intersects the line $y=x-4k+2$ but that does not cross the line $y=x-4k-1$.

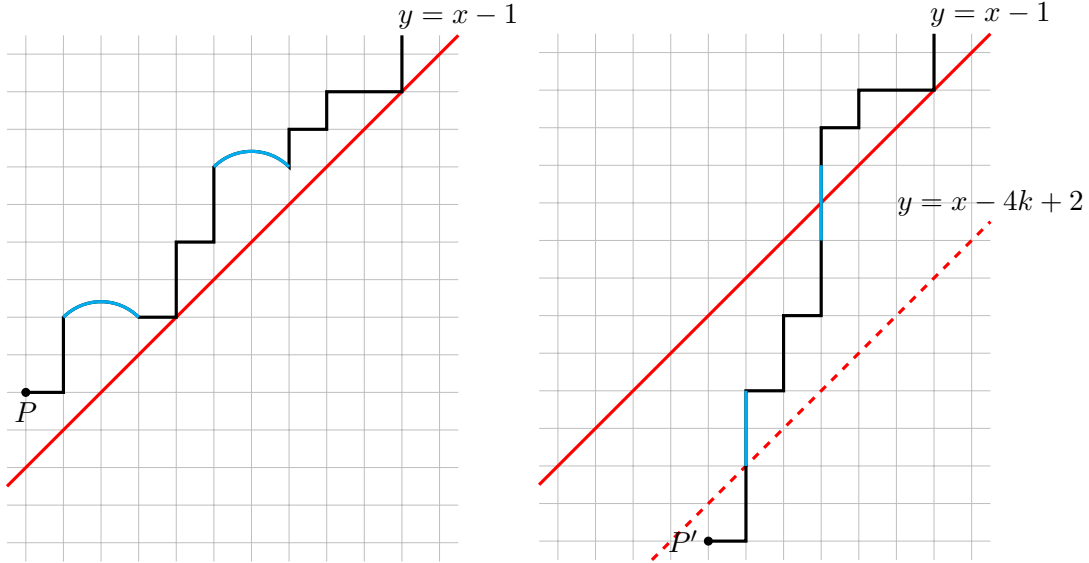
\begin{figure}
\begin{tikzpicture}
\begin{scope}[scale=.5]
	\draw[help lines,step=1cm, lightgray] (-.5,-2.5) grid (11.5,11.5);
	\draw[very thick, red] (-.5,-.5) -- (11.5,11.5);
	\draw[very thick] (0,2) --++ (1,0) --++ (0,2) to[out=45, in=135] ++(2,0) --++ (1,0) --++ (0,2) --++ (1,0) --++ (0,2) to[out=45, in=135] ++ (2,0) --++ (0,1) --++ (1,0) --++ (0,1) --++ (2,0) --++ (0,1.5);
	\filldraw[black] (0,2) circle (3pt) node[below]{$P$};	
	\node at (11.5,12) {$y=x-1$};
	\draw[very thick,cyan] (1,4) to[out=45, in=135] ++(2,0);
	\draw[very thick,cyan] (5,8) to[out=45, in=135] ++(2,0);

\end{scope}
\end{tikzpicture}
\begin{tikzpicture}
\begin{scope}[scale=.5]
	\draw[help lines,step=1cm, lightgray] (-.5,-2.5) grid (11.5,11.5);
	\draw[very thick, red] (-.5,-.5)-- (11.5,11.5);
	\draw[very thick, red, dashed] (2.5,-2.5)-- (11.5,6.5);
	\draw[very thick] (4,-2) --++ (1,0) --++ (0,2) --++(0,2) --++ (1,0) --++ (0,2) --++ (1,0) --++ (0,2) --++ (0,2) --++ (0,1) --++ (1,0) --++ (0,1) --++ (2,0) --++ (0,1.5);
	\filldraw[black] (4,-2) circle (3pt) node[left]{$P'$};
	\node at (11.5,12) {$y=x-1$};
	\node at (11.5,7) {$y=x-4k+2$};
	\draw[very thick,cyan] (5,0) --++(0,2);
	\draw[very thick,cyan] (7,6) --++(0,2);
\end{scope}
\end{tikzpicture}
\caption{Situation in Lemma~\ref{lem:kOrthogonalPaths} when replacing $\o$-horizontal steps by vertical steps of length~$2$.}
\label{fig:lemmaso}
\end{figure}

	 This mapping is again a weight-preserving bijection. We continue by using the modified reflection principle in Lemma~\ref{lem:modifiedreflection}. 
	 
	 In order to compute the generating function of paths from $(a+2k,b-2k)$ to $(c,2n+a+b-c)$ that intersect the line $y=x-4k+2$, we have to distinguish two cases. If $(a+2k,b-2k)$ already lies below $y=x-4k+2$, that is, if $b-a=0$, then the generating function is simply $e_{c-a-2k}(\x^\pm)$. Otherwise, we reflect $(a+2k,b-2k)$ along $y=x-4k+2$ and obtain $(b+2k-2,a-2k+2)$. Thus, the generating function of these paths is $e_{c-b-2k+2}(\x^\pm)$. 
	 
	 Lattice paths from $(a+2k,b-2k)$ to $(c,2n+a+b-c)$ which cross the line $y=x-4k-1$ touch also the line $y=x-4k-2$. Reflecting $(a+2k,b-2k)$ along  $y=x-4k-2$ gives $(b+2k+2,a-2k-2)$, and thus the generating function of these paths is $e_{c-b-2k-2}(\x^\pm)$. This completes the proof. 
\end{proof}

Using Lemma~\ref{lem:SymplecticPaths}, we know that the generating function of lattice paths from $(a,b)$ to $(c,2n+a+b-c)$ that do not go below the line $y=x-1$ and have no $\o$-horizontal steps is $e_{c-a}(\x^\pm)-e_{c-b-2}(\x^\pm)$. By summing over all possible numbers of $\o$-horizontal steps, we finally obtain in the case $b-a \geq 2$:
\begin{equation*}
	\big( e_{c-a}(\x^\pm)-e_{c-b-2}(\x^\pm) \big) + \sum_{k \geq 1} \big( e_{c-b-2k+2}(\x^\pm)-e_{c-b-2k-2}(\x^\pm) \big) = e_{c-a}(\x^\pm)+e_{c-b}(\x^\pm).
\end{equation*}
On the other hand, if $b-a=0$, we obtain
\begin{equation*}
	\big( e_{c-a}(\x^\pm)-e_{c-a-2}(\x^\pm) \big) + \sum_{k \geq 1} \big( e_{c-a-2k}(\x^\pm)-e_{c-a-2k-2}(\x^\pm) \big) = e_{c-a}(\x^\pm).
\end{equation*}

\begin{cor}\label{cor:OrthogonalPaths}
Let $a,b,c\in\mathbb{Z}$ and $n\in\mathbb{N}$ such that $a+b$ is even and the points $(a,b)$ and $(c,2n+a+b-c)$ lie strictly above the line $y=x-1$. Then the generating function of lattice paths from $(a,b)$ to $(c,2n+a+b-c)$ with step set $\{(1,0),(0,1),(2,0)\}$ that do not cross the line $y=x-1$ and where $\o$-horizontal steps are only allowed to end on the line $y=x$ is
	\begin{equation*}
		\begin{cases}
		e_{c-a}(\x^\pm) & \text{if $b-a=0$,} \\
		e_{c-a}(\x^\pm)+e_{c-b}(\x^\pm) & \text{if $b-a \geq 2$,}
		\end{cases}
	\end{equation*}	
where $\x=(x_1,x_1^{-1},\dots,x_n,x_n^{-1})$.
\end{cor}

We now apply the Lindstr\"om--Gessel--Viennot Lemma~\ref{lem:LGV} to all lattice paths from $(\mu'_j-j+1,2m-\mu'_j+j-1)$ to $(\lambda'_i-i+1,2n+2m-\la_i'+i-1)$ for $1 \leq i,j \leq N$ with $\la_1\leq N$, $l(\mu)\leq m$, $l(\la)\leq n+m$ and step set $\{(1,0),(0,1),(2,0)\}$ such that the paths weakly stay above the line $y=x-1$ and $\o$-horizontal steps are only allowed to end on the line $y=x$.
By Corollary~\ref{cor:OrthogonalPaths} this gives
\begin{equation}\label{eq:SkewOrthogonalJT}
	\frac{1}{2^{[m=l(\mu)]}} \det_{1 \leq i,j \leq N} \big( e_{\lambda'_i -\mu'_j-i+j}(\x^\pm) + e_{\lambda'_i + \mu'_j-i-j-2m+2}(\x^\pm) \big).
\end{equation}
Note that the prefactor $\frac{1}{2}$ takes care of the case $m=l(\mu)$ because then the first column of the underlying matrix in the previous determinant is $(2 e_{\lambda'_1-\mu'_1}(\x^\pm),2 e_{\lambda'_2-\mu'_1-1}(\x^\pm),\dots,\allowbreak 2 e_{\lambda'_N-\mu'_1-N+1}(\x^\pm))^{\top}$.

To complete the proof, we want to show that \eqref{eq:SkewOrthogonalJT} indeed represents a formula for $\o^m_{\lambda/\mu}(\x)$. This is done by a sign-reversing involution. First, we recall that the Lindstr\"om--Gessel--Viennot Lemma~\ref{lem:LGV} \eqref{eq:SkewOrthogonalJT} yields a signed enumeration of weakly non-intersecting lattice paths where $\o$-horizontal steps might intersect with vertical steps. Figure~\ref{fig:OrthogonalPathsInvolution} shows an example of lattice paths with such  intersections. On the other hand, there might still be trapped positions. The sign-reversing involution will transform an intersection into a trapped position, or vice versa.

\begin{figure}[htb]
	\centering
\begin{tikzpicture}[scale=.5,baseline=(current bounding box.center)]
	\draw [help lines,step=1cm,lightgray] (-3.75,1.25) grid (8.75,10.75);
	\draw[red,very thick] (2.25,1.25) -- (8.75,7.75); 
	
	\filldraw[black] (2,2) circle (3pt);
	\draw[very thick] (2,2) --++ (1,0) --++ (0,1) --++ (0,1) --++ (1,0) --++ (0,1) --++ (0,1) --++ (1,0) --++ (0,1);
	\draw[very thick] (5,7) to[out=45, in=135] (7,7);
	\draw[very thick] (7,7) --++ (1,0) --++ (0,1) --++ (0,1) --++ (0,1) --++ (0,.75);
	
	\filldraw[black] (0,4) circle (3pt);
	\draw[very thick] (0,4) --++ (1,0) --++ (1,0) --++ (0,1) --++ (1,0);
	\draw[very thick] (3,5) to[out=45, in=135] (5,5);
	\draw[very thick] (5,5) --++ (1,0)  --++ (0,1) --++ (0,1) --++ (0,1) --++ (1,0) --++ (0,1) --++ (0,1) --++ (0,.75);
	
	\filldraw[black] (-1,5) circle (3pt);
	\draw[very thick] (-1,5) --++ (0,1) --++ (1,0) --++ (1,0) --++ (1,0) --++ (1,0) --++ (0,1) --++ (0,1) --++ (1,0) --++ (1,0) --++ (0,1) --++ (0,1) --++ (0,.75);
	
	\filldraw[black] (-3,7) circle (3pt);
	\draw[very thick] (-3,7) --++ (1,0) --++ (1,0) --++ (1,0) --++ (1,0) --++ (1,0) --++ (0,1) --++ (0,1) --++ (1,0) --++ (0,1) --++ (1,0) --++ (0,.75);
\end{tikzpicture}
\caption{Weakly non-intersecting lattice paths that are enumerated by the Lindstr\"om--Gessel--Viennot Lemma.}
\label{fig:OrthogonalPathsInvolution}
\end{figure}
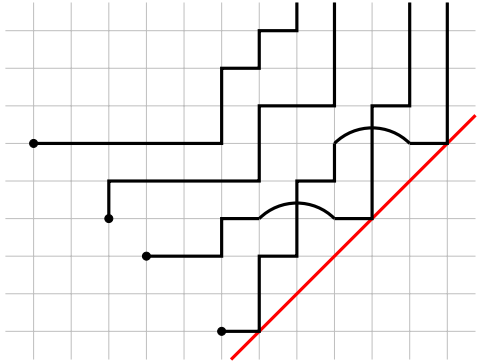

The intersections between $\o$-horizontal steps and a pair of unit vertical steps only occur along the line $y=x+1$. Consider such an intersection and assume it is at the point $(d,d+1)$. That means we have an $\o$-horizontal step $(d-1,d+1) \to (d+1,d+1)$ in one path, and two unit vertical steps $(d,d) \to (d,d+1) \to (d,d+2)$ in the other path. We perform the following local changes to the family of lattice paths along the line $y+x=2d+1$ as 
indicated in Figure~\ref{fig:OrthogonalPathsLocalChanges}: First we resolve the crossing by replacing it by the following two turns $(d,d) \to (d+1,d) \to (d+1,d+1)$ and $(d-1,d+1) \to (d,d+1) \to (d,d+2)$, which will then 
appear in different paths. Note that this changes the sign of the permutation, however, it does not change the weight.
Then we look for the unoccupied point $(x_0,y_0)$ on the line $y+x=2d+1$ that is above the line $y=x-1$ and has minimal $y_0-x_0$. This point is unique and between that point and the line $y=x-1$ there is a series of left turns along the line $y+x=2d+1$. We take the left turn $(x_0,y_0-1) \to (x_0+1,y_0-1) \to (x_0+1,y_0)$ and replace it by $(x_0,y_0-1) \to (x_0,y_0) \to (x_0+1,y_0)$.

\begin{figure}[htb]
	\centering
	\begin{tikzpicture}[scale=.5,baseline=(current bounding box.center)]
		\draw [help lines,step=1cm,lightgray] (.25,-.75) grid (5.75,4.75);
		\draw[very thick, red] (4.25,-.75) -- (5.75,.75); 
		
		\draw[very thick] (3,1) to[out=45, in=135] (5,1);
		
		\draw[very thick] (4,0) -- (4,1) -- (4,2);
		
		\draw[very thick] (2,2) -- (3,2) -- (3,3);
		
		\draw[very thick] (1,3) -- (2,3) -- (2,4);	
		\filldraw[red] (2,3) circle (3pt);
	\end{tikzpicture}
	\quad$\longleftrightarrow$\quad
	\begin{tikzpicture}[scale=.5,baseline=(current bounding box.center)]
		\draw [help lines,step=1cm,lightgray] (.25,-.75) grid (5.75,4.75);
		\draw[very thick, red] (4.25,-.75) -- (5.75,.75); 
		
		\draw[very thick] (3,1) -- (4,1) -- (4,2);
		
		\draw[very thick] (4,0) -- (5,0) -- (5,1);
		
		\draw[very thick] (2,2) -- (3,2) -- (3,3);
		
		\draw[very thick] (1,3) -- (1,4) -- (2,4);	
		
		\filldraw[red] (2,3) circle (3pt);
	\end{tikzpicture}
	\caption{\label{fig:OrthogonalPathsLocalChanges} Local changes between lattice paths with an intersection of an $\o$-horizontal step with two vertical steps (left) and with a trapped position (right).}
\end{figure}
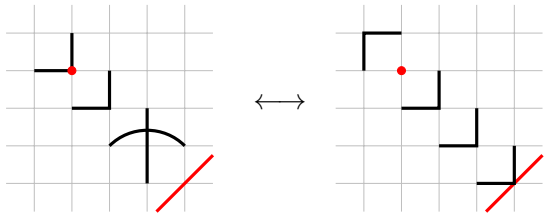

Note that this local transformation only changes the sign of the weight. In addition, we see that the resulting constellations such as on the right of Figure~\ref{fig:OrthogonalPathsLocalChanges} are exactly the trapped positions that cannot occur in the lattice path interpretation of skew $(n,m)$-even orthogonal tableaux due to the $m$-even orthogonal condition. This observation leads us to the announced sign-reversing involution.

Consider the families of lattice paths whose signed enumeration is given by \eqref{eq:SkewOrthogonalJT}. If it contains a crossing of an $\o$-horizontal step with two vertical steps or if it contains a trapped position
then choose a canonical occurrence as follows.
Both the crossings and the trapped position are located on lines $y+x=2d+1$ for integers $d$. We consider the crossing or the trapped position for which the $d$ is minimal. Then we perform the local changes as exemplified in Figure~\ref{fig:OrthogonalPathsLocalChanges}. This mapping changes the sign of the weight and is readily invertible. We are left with families of lattice paths that have neither crossings nor trapped positions, which are exactly those that correspond to skew $(n,m)$-even orthogonal tableaux. This finally proves that
\begin{equation*}
	\o^m_{\lambda/\mu}(\x)=\frac{1}{2^{[m=l(\mu)]}} \det_{1 \leq i,j \leq N} \big( e_{\lambda'_i -\mu'_j-i+j}(\x^\pm) + e_{\lambda'_i + \mu'_j-i-j-2m+2}(\x^\pm) \big), 
\end{equation*} 
and concludes the combinatorial proof of Theorem~\ref{thm:dual} \eqref{dual_o}.

\section{Proofs of the Jacobi--Trudi formulae} 
\label{sec:JacobiTrudi}
The purpose of this section is to derive ordinary Jacobi--Trudi-type formulae
for the characters $\sp^m_{\la/\mu}$, $\so^m_{\la/\mu}$ and $\o^m_{\la/\mu}$.
We achieve this by an algebraic approach using complementary cofactors,
which is one of the standard ways to prove the equivalence of the 
Jacobi--Trudi formula for Schur functions and its dual.
These computations all work in the ring of symmetric functions on a countable
alphabet $\X=(x_1,x_2,x_3,\dots)$, and so in this section we work in such
generality.
The formulae for the actual characters are obtained by substituting 
$\x^\pm=(x_1,x_1^{-1},\dots,x_n,x_n^{-1})$ for $\X$.

Given an $N\times N$ matrix $A$ and a pair of sequences
$\sigma,\tau$ of the same length, we let 
$A^\sigma_\tau$ denote the matrix obtained by extracting rows and columns with 
indices from $\sigma$ and $\tau$ respectively in the given order.
The basic lemma we need is the following \cite[Lemma~A.42]{FulHar91}.
\begin{lem}\label{Lem_cofactor}
Let $A$ and $B$ be $N\times N$ mutually inverse matrices
and let $(\sigma,\sigma')$ and $(\tau,\tau')$ be pairs of complementary 
subsets of $\{1,\dots,N\}$ such that $\abs{\sigma}=\abs{\tau}$.
Then
\[
\det(A^\sigma_\tau)=\varepsilon \det(A)\det\big(B^{\tau'}_{\sigma'}\big),
\]
where $\varepsilon$ is the product of the signs of the permutations
formed by the words $\sigma\sigma'$ and $\tau\tau'$.
\end{lem}

Recall that the matrices
\[
\big(e_{i-j}(\X)\big)_{1\leq i,j\leq r}
\quad\text{and}\quad
\big((-1)^{i-j}h_{i-j}(\X)\big)_{1\leq i,j\leq r},
\]
are lower-unitriangular, i.e., lower-triangular with diagonal entries all $1$, and where as before
$\X=(x_1,x_2,x_3,\dots)$. Moreover, thanks to
the relationship \cite[p.~21]{Macdonald}
\begin{equation}\label{Eq_convolution}
\sum_{k=0}^r (-1)^{k}e_{r-k}(\X)h_{k}(\X)=\delta_{r,0}
\end{equation}
where $\delta_{a,b}$ is the usual Kronecker delta, they are mutually inverse.
We actually require the following slightly more general pair of matrices.
\begin{lem}
For $m,N\in\mathbb{N}$, $k\in\mathbb{Z}$ and a parameter $t$ the matrices 
\[
\mathcal{E}(N,m,k;t)
\coloneqq \big(e_{i-j}(\X)
+\left[j<m+\lceil k/2\rceil\right]te_{i+j-2m-k}(\X)\big)_{1\leq i,j\leq N}
\]
and
\[
\mathcal{H}(N,m,k;t) \coloneqq \big((-1)^{i-j}\big(h_{i-j}(\X)
-\left[i>m+\lfloor k/2\rfloor\right]
(-1)^kth_{2m-i-j+k}(\X)\big)\big)_{1\leq i,j\leq N}
\]
are lower-unitriangular and mutually inverse.
\end{lem}
\begin{proof}
The $(i,j)$-th entry in the product of the two matrices is
\begin{multline*}
\sum_{\ell=0}^N(-1)^{\ell-j}e_{i-\ell}(\X)h_{\ell-j}(\X)
+\sum_{\ell=0}^{m+\lceil k/2\rceil-1}(-1)^{\ell-j}te_{i+\ell-2m-k}(\X)
h_{\ell-j}(\X)\\
-\sum_{\ell=m+\lfloor k/2\rfloor +1}^N
(-1)^{\ell-j+k}te_{i-\ell}(\X)h_{2m-\ell-j+k}(\X),
\end{multline*}
where we assume that $i\geq j$ since the lower-triangularity is clear from 
the definition.
The first sum in this expression simplifies to
\[
\sum_{\ell=0}^N(-1)^{\ell-j}e_{i-\ell}(\X)h_{\ell-j}(\X)
=\sum_{\ell=0}^{i-j}(-1)^\ell e_{i-j-\ell}(\X)h_\ell(\X)=\delta_{i-j,0},
\]
by \eqref{Eq_convolution}.
For the remaining sums, we obtain by substituting $\ell\mapsto m-\ell+\lceil k/2\rceil-1$ in the first sum and $\ell\mapsto m+\ell+\lfloor k/2\rfloor+1$ in the second sum
\begin{align*}
&\sum_{\ell=0}^{m+\lceil k/2\rceil-1}\!(-1)^{\ell-j}te_{i+\ell-2m-k}(\X)
h_{\ell-j}(\X)
-\!\sum_{\ell=m+\lfloor k/2\rfloor +1}^N\!
(-1)^{\ell-j+k}te_{i-\ell}(\X)h_{2m-\ell-j+k}(\X) \\
&\qquad=
\sum_{\ell=0}^{m+\lceil k/2\rceil-1}(-1)^{\ell-j+\lceil k/2\rceil+m+1}
te_{i-\ell-m-\lfloor k/2\rfloor-1}(\X)h_{m+\lceil k/2\rceil-\ell-j-1}(\X)\\
&\qquad\qquad-\sum_{\ell=0}^{N-m-\lfloor k/2\rfloor-1}
(-1)^{\ell-j+\lceil k/2\rceil+m+1}te_{i-\ell-m-\lfloor k/2\rfloor-1}(\X)
h_{m+\lceil k/2\rceil-\ell-j-1}(\X)\\
&\qquad=0,
\end{align*}
where the last equality follows since the summands are equal and vanish unless 
the index satisfies
$\ell\leq\min\{m+\lceil k/2\rceil-1,N-m-\lfloor k/2\rfloor-1\}$.
\end{proof}

Recall from \cite[p.~3]{Macdonald} that for a partition $\la\subseteq (N^M)$ 
the sets
\begin{equation}
\label{eq:disjoint sets}
\{\la_i+M-i+1:1\leq i\leq M\} \quad\text{and}\quad
\{M+j-\la_j':1\leq j\leq N\}
\end{equation}
form a disjoint union of $\{1,\dots,N+M\}$.
This implies that for partitions $\mu\subseteq\la\subseteq(N^M)$ 
the sequences
\begin{align*}
(\sigma,\sigma')& \coloneqq
(\la_1'+N,\dots,\la_N'+1,N-\la_1+1,\dots,N+M-\la_M)\\
(\tau,\tau')& \coloneqq (\mu_1'+N,\dots,\mu_N'+1,N-\mu_1+1,\dots,N+M-\mu_M).
\end{align*}
form permutations of $\{1,\dots,N+M\}$.
Applying Lemma~\ref{Lem_cofactor} with these choices of $(\sigma,\sigma')$ and 
$(\tau,\tau')$,
$A=\mathcal{E}(N+M,m,k;t)$ and $B=\mathcal{H}(N+M,m,k;t)$ we obtain
that the determinants
\begin{subequations}\label{Eq_general-dual}
\begin{equation}\label{Eq_general-duala}
\det_{1\leq i,j\leq N}\big(
e_{\la_i'-\mu_j'-i+j}(\X)+[N+\mu_j'-j+1<m+\lceil k/2\rceil]
te_{\la_i'+\mu_j'-i-j+2(N-m+1)-k}(\X)\big)
\end{equation}
and
\begin{equation}\label{Eq_general-dualb}
\det_{1\leq i,j\leq M}\big(
h_{\la_i-\mu_j-i+j}(\X)
-[N-\mu_j+j>m+\lfloor k/2\rfloor](-1)^k
th_{\la_i+\mu_j-i-j+2(m-N)+k}(\X)\big),
\end{equation}
\end{subequations}
are equal.
Note that in this case the product of the signs of the permutations 
is $\varepsilon=(-1)^{\abs{\la}+\abs{\mu}}$, which is a consequence of the proof that \eqref{eq:disjoint sets} form a disjoint union of
$\{1,\dots,N+M\}$.
If one sets $t=0$ in \eqref{Eq_general-dual} then 
\[
\det_{1\leq i,j\leq N}\big(e_{\la_i'-\mu_j'-i+j}(\X)\big)
=\det_{1\leq i,j\leq M}\big(h_{\la_i-\mu_j-i+j}(\X)\big),
\]
showing that the Jacobi--Trudi formula and its dual are equal.
The dual forms of our Jacobi--Trudi formulae for the skew characters
$\sp^m_{\la/\mu}, \so^m_{\la/\mu}$ and $\o^m_{\la/\mu}$ 
are similarly contained in \eqref{Eq_general-dual}.
We can now prove Theorem~\ref{thm:hformulae}.

\begin{proof}[Proof of Theorem~\ref{thm:hformulae}]
Fix partitions $\mu\subseteq\la$ such that $l(\mu)\leq m$,
$l(\la)\leq n+m$ and $\la\subseteq(N^M)$.
For the symplectic case we set $(m,k,t)\mapsto(N+m,2,-1)$ in 
\eqref{Eq_general-dual}, which gives
\begin{multline*}
\det_{1\leq i,j\leq N}\big(
e_{\la_i'-\mu_j'-i+j}(\X)-[\mu_j'-j<m]e_{\la_i'+\mu_j'-i-j-2m}(\X)\big)\\
=\det_{1\leq i,j\leq M}\big(
h_{\la_i-\mu_j-i+j}(\X)+[j-\mu_j>m+1]h_{\la_i+\mu_j-i-j+2m+2}(\X)\big).
\end{multline*}
Since $\mu_j'\leq m$ for $1\leq j\leq N$ we always have two terms in 
each entry of the determinant on the left. 
After replacing $\X$ by $(x_1,x_1^{-1},\dots,x_n,x_n^{-1})$ it is equal to 
$\sp_{\la/\mu}^m$ by \eqref{dual_sp}.
Further, $m+\mu_j-j+1<0$ is only true for $j>m+1$, so the determinant on
the right-hand side becomes \eqref{Eq_sph} with $N\mapsto M$ and the same 
variable substitution.

Instead taking $(m,k,t)\mapsto(N+m,1,1)$ in \eqref{Eq_general-dual}
produces
\begin{multline*}
\det_{1\leq i,j\leq N}\big(
e_{\la_i'-\mu_j'-i+j}(\X)+[\mu_j'-j<m]e_{\la_i'+\mu_j'-i-j-2m+1}(\X)\big)\\
=\det_{1\leq i,j\leq M}\big(
h_{\la_i-\mu_j-i+j}(\X)+[j-\mu_j>m]h_{\la_i+\mu_j-i-j+2m+1}(\X)\big).
\end{multline*}
By the same arguments as above the left-hand side is equal to $\so_{\la/\mu}^m$
and the right-hand side to \eqref{Eq_soh} with $N\mapsto M$,
after appropriate substitution of
variables.

Finally, choosing $(m,k,t)\mapsto(N+m,0,1)$ in \eqref{Eq_general-dual}
yields
\begin{multline*}
\det_{1\leq i,j\leq N}\big(
e_{\la_i'-\mu_j'-i+j}(\X)+[\mu_j'-j<m-1]e_{\la_i'+\mu_j'-i-j-2m+2}(\X)\big)\\
=\det_{1\leq i,j\leq M}\big(
h_{\la_i-\mu_j-i+j}(\X)-[j-\mu_j>m]h_{\la_i+\mu_j-i-j+2m}(\X)\big).
\end{multline*}
The right-hand clearly gives \eqref{Eq_oh} with $N\mapsto M$.
If $m=l(\mu)=\mu_1'$ then the $(1,j)$-th entry of the determinant on the left-hand 
side is equal to $e_{\la_i'-\mu_j'-i+j}(\X)$ which is half of the $(1,j)$-th entry of \eqref{dual_o}. Hence by factoring out $\frac{1}{2}$  from the first column, the left-hand side of the above equation is equal to \eqref{dual_o}.
\end{proof}

In \cite{FulKra97}, Fulmek and Krattenthaler derive the Jacobi--Trudi formulae for 
$\sp_\la$ and $\o_\la$ using Gessel and Viennot's technique of dual paths, 
which is, in essence, a combinatorial realisation of the approach of this 
section.

\section{Combinatorial proofs of the Giambelli formulae} 
\label{sec:Giambelli}
Fulmek and Krattenthaler \cite{FulKra97} provided combinatorial proofs of the  Giambelli identities for ordinary symplectic and odd orthogonal characters based on Stembridge's proof of the Giambelli identity for ordinary Schur functions \cite{Ste90}. We adapt their ideas in order to prove combinatorially the respective skew Giambelli identities analogous to formula \eqref{eq:ordinarySkewGiambelli}. By means of a sign-reversing involution similar to the one in Section~\ref{sec:dualJacobiTrudi_o}, we are also able to provide a combinatorial proof of the Giambelli identity for skew even orthogonal characters.

In Section~\ref{sec:tableaux}, we obtained the lattice path models (families of $(n,m)$-symplectic/odd orthogonal/even orthogonal paths) that we used to prove the dual Jacobi--Trudi formulae by reading off the corresponding tableaux columnwise. Ordinary Giambelli-type identities express characters of arbitrary shapes as determinants whose entries are characters of hook shapes. That observation suggests to read off tableaux hookwise.

Consider the Young diagram of shape $\lambda/\mu$ for two partitions $\mu \subseteq \lambda$. The \defn{hook} at position $(i,j)$ consists of the cells
\begin{itemize}
	\item $(i,j) \cup \{(i,k) : j+1 \leq k \leq \lambda_i \} \cup \{(k,j) : i+1 \leq k \leq \lambda'_j \}$ if $(i,j) \in \lambda/\mu$ and
	\item $\{(i,k) : \mu_i + 1 \leq k \leq \lambda_i \} \cup \{(k,j) : \mu'_j + 1 \leq k \leq \lambda'_j \}$ if $(i,j) \notin \lambda/\mu$.
\end{itemize}
In other words, the hook at a position $(i,j)$ is the union of the cell $(i,j)$ (the \defn{corner} of the hook) together with the cells to the right of it in the $i$-th row (the \defn{arm} of the hook) and with the cells below of it in the $j$-th column (the \defn{leg} of the hook) provided that $(i,j)$ is part of the Young diagram of shape $\lambda/\mu$. If $(i,j)$ is not an element of the skew shape~$\lambda/\mu$, that is, $(i,j) \notin \lambda/\mu$, then the hook is not connected and we call it \defn{broken}. Figure~\ref{fig:hooks} illustrates the two different types of hooks.

\begin{figure}[htb]
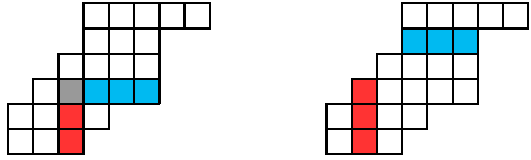

	\centering
	\ytableausetup{smalltableaux}
	\begin{ytableau}
	\none & \none & \none & \empty & \empty & \empty & \empty & \empty \\
	\none & \none & \none & \empty & \empty & \empty \\
	\none & \none & \empty & \empty & \empty & \empty \\
	\none & \empty & *(gray!70)\bullet & *(cyan!70)A & *(cyan!70)A & *(cyan!70)A \\
	\empty & \empty & *(red!70)L & \empty \\
	\empty & \empty & *(red!70)L
	\end{ytableau}
	\qquad\qquad
	\begin{ytableau}
	\none & \none & \none & \empty & \empty & \empty & \empty & \empty \\
	\none & \none[\scriptstyle{\bullet}] & \none & *(cyan!70)A & *(cyan!70)A & *(cyan!70)A\\
	\none & \none & \empty & \empty & \empty & \empty \\
	\none & *(red!70)L & \empty & \empty & \empty & \empty \\
	\empty & *(red!70)L & \empty & \empty \\
	\empty & *(red!70)L & \empty
	\end{ytableau}
	\ytableausetup{nosmalltableaux}
	\caption{A Young diagram of shape $(8,6,6,6,4,3)/(3,3,2,1)$ with the hook at position $(4,3)$ (left) and with the broken hook at position $(2,2)$ (right) marked. The arms of the hooks are shaded blue with inscribed $A$'s and the legs shaded red with inscribed $L$'s.}
	\label{fig:hooks}
\end{figure}

Next, we introduce the set-up common to all three considered models for the lattice paths that we obtain by reading off tableaux hookwise; for the specific details of each model see the following subsections. Let $\lambda=(\alpha_1,\dots,\alpha_p \vert \beta_1,\dots,\beta_p)$ and $\mu=(\gamma_1,\dots,\allowbreak \gamma_q \vert \allowbreak \delta_1,\dots,\delta_q)$ be two partitions $\mu \subseteq \lambda$ with $l(\mu)\leq m$ and $l(\la)\leq n+m$. Define $A_i \coloneqq (-\alpha_i,2n+2m-1)$ and $B_i \coloneqq (\beta_i +1,2n+2m-\beta_i-1)$ for $1 \leq i \leq p$ and $C_i \coloneqq (-\gamma_i,2m)$ and $D_i \coloneqq (\delta_i +1,2m-\delta_i-1)$ for $1 \leq i \leq q$.

We will associate lattice paths corresponding to the \defn{principal hooks} of the skew tableaux: the $i$-th principal hook is the hook at the diagonal position $(i,i)$. For $q+1 \leq i \leq p$, the paths from $A_i$ to $B_i$ correspond to the $i$-th principal hook read from right to left and from top to bottom. For $1 \leq i \leq q$, the $i$-th principal hook is broken. The paths from $A_i$ to $C_i$ correspond to the arm of the $i$-th principal hook read from right to left, whereas the paths from $D_i$ to $B_i$ correspond to the leg of the $i$-th principal hook read from top to bottom. 

The lattice paths fulfil the following properties:
\begin{itemize}
	\item In the region $\{(x,y)\in\mathbb{Z}^2 : x \leq 0\}$, lattice paths stay weakly above the line~$x=2m$ and horizontal unit steps $(i,2m+j) \to (i+1,2m+j)$ get the weight~$x_{(j+2)/2}^{-1}$ if $j$ is even and $x_{(j+1)/2}$ if $j$ is odd. 
	\item In the region $\{ (x,y)\in\mathbb{Z}^2 : x \geq 0\}$, lattice paths stay weakly above the lines $y=-x+2m$ and $y=x-1$ and horizontal unit steps $(i,2m+j) \to (i+1,2m+j)$ are assigned the weight~$x_{(i+j+2)/2}^{-1}$ if $i+j$ is even and $x_{(i+j+1)/2}$ if $i+j$ is odd.
\end{itemize}
	All other steps have weight~$1$, and the weight of a family of paths is the product of the weights of all its steps. Note that this set-up constitutes a combined $e$-  and $h$-labelling: We have an $e$-labelling in the region $x > 0$ for the legs of the hooks and an $h$-labelling in the region $x \leq 0$ for the arms and the corners of the hooks. The exact lattice path model for each of the skew characters is specified in the respective section.

A crucial observation will be be the following: For each of the different lattice path models, the families of $p+q$ lattice paths from $\{ A_1,\dots,A_p,D_1,\dots,D_q\}$ to $\{ B_1,\dots,B_p, \allowbreak C_1, \dots,C_q\}$ are strongly non-intersecting if and only if we have paths from $A_i$ to $C_i$ as well as from $D_i$ to $B_i$ for $1 \leq i \leq q$ and from $A_i$ to $B_i$ for $q+1 \leq i \leq p$. By applying the Lindstr\"{o}m--Gessel--Viennot Lemma~\ref{lem:LGV} and --- in the case of skew even orthogonal characters --- by a sign-reversing involution, we will show that
\begin{equation}\label{eq:GiambelliGesselViennotDet}
	(-1)^q \det \begin{pmatrix}
		(\mathcal{P} (A_i \to B_j))_{1 \leq i,j \leq p} & ( \mathcal{P} (A_i \to C_j))_{\substack{1 \leq i \leq p,\\1 \leq j \leq q}}\\
		( \mathcal{P} (D_i \to B_j))_{\substack{1 \leq i \leq q,\\1 \leq j \leq p}} & ( \mathcal{P} (D_i \to C_j) )_{1 \leq i,j \leq q}
	\end{pmatrix}
\end{equation}
yields the respective Giambelli-type formulae, where $(-1)^q$ is the sign of the permutation
\begin{equation*}
	\begin{pmatrix}
		1 & \cdots & q & q+1 & \cdots & p & p+1 & \cdots & p+q\\
		p+1 & \cdots & p+q & q+1 & \cdots & p & 1 & \cdots & q
	\end{pmatrix}.
\end{equation*}
The sign can be readily computed since the permutation is equal to 
\begin{equation*}
	(1 \quad p+1) (2 \quad p+2) \cdots (q \quad p+q)
\end{equation*}
in cycle notation, which is a product of exactly $q$ transpositions.

In the following three subsections, we prove Theorem~\ref{thm:giambelli}.

\subsection{Proof of the skew symplectic Giambelli identity~(\ref{Eq_spgiambelli})}

A skew $(n,m)$-symplectic tableau is encoded by non-intersecting lattice paths from $A_i$ to $C_i$ and from $D_i$ to $B_i$ for $1 \leq i \leq q$ as well as from $A_i$ to $B_i$ for $q+1 \leq i \leq p$ with step set $\{ (1,0),(0,-1)\}$ in the region $x \leq 0$ and step set $\{ (1,0),(0,1)\}$ in the region $x \geq 0$ that stay all weakly above the line $y=x-1$; see Figure~\ref{fig:SkewSymplecticGiambelli} for an example.
\begin{figure}[htb]
	\centering
	\begin{ytableau}
		\none & \none & \none & 1 & \ov{3} & \ov{4}\\
		\none & \none & \ov{1} & 2 & 3\\
		\ov{1} & \ov{2} & 2 & 3 \\
		2 & \ov{3} & \ov{4} & \ov{4} \\
		\ov{3} 
	\end{ytableau}
	\qquad\qquad
	\begin{tikzpicture}[scale=.5,baseline=(current bounding box.center)]
		\draw [help lines,step=1cm,lightgray] (-5.75,-.75) grid (5.75,11.75);
		
		\draw[very thick, red] (0.25,-.75) -- (5.75,4.75);
		
		\filldraw[black] (-2,4) circle (3pt);
		\filldraw[black] (0,4) circle (3pt);
		\filldraw[black] (2,2) circle (3pt);
		\filldraw[black] (1,3) circle (3pt);
		\filldraw[black] (-5,11) circle (3pt);
		\filldraw[black] (-3,11) circle (3pt);
		\filldraw[black] (-1,11) circle (3pt);
		\filldraw[black] (0,11) circle (3pt);
		\filldraw[black] (1,11) circle (3pt);
		\filldraw[black] (2,10) circle (3pt);
		\filldraw[black] (3,9) circle (3pt);
		\filldraw[black] (5,7) circle (3pt);
		
		\draw[very thick] (-5,11) node[above]{$A_1$} -- (-5,10) --node[above]{$\ov{4}$} (-4,10) -- (-4,8) --node[above]{$\ov{3}$} (-3,8) -- (-3,5) --node[above]{$1$} (-2,5) -- (-2,4) node[below]{$C_1$};
		\draw[very thick] (-3,11) node[above]{$A_2$} -- (-3,9) --node[above]{$3$} (-2,9) -- (-2,7) --node[above]{$2$} (-1,7) -- (-1,4) --node[above]{$\ov{1}$} (0,4) node[below]{$C_2$};
		\draw[very thick] (-1,11) node[above]{$A_3$} -- (-1,9) --node[above]{$3$} (0,9) --(0,7) --node[above]{$2$} (1,7) -- (1,9) --node[above]{$\ov{4}$} (2,9) -- (2,10) node[above]{$B_3$};
		\draw[very thick] (0,11) node[above]{$A_4$} -- (0,10) --node[above]{$\ov{4}$} (1,10) -- (1,11) node[above]{$B_4$};
		\draw[very thick] (1,3) node[below left]{$D_2$} -- (1,5) --node[above]{$\ov{2}$} (2,5) -- (2,6) --node[above]{$\ov{3}$} (3,6) -- (3,9) node[above]{$B_2$};
		\draw[very thick] (2,2) node[below left]{$D_1$} --node[above]{$\ov{1}$} (3,2) -- (3,4) --node[above]{$2$} (4,4) --node[above]{$\ov{3}$} (5,4) -- (5,7) node[above]{$B_1$};
		
		\filldraw[red] (0,0) circle (3pt) node[left]{\color{black}$(0,0)$};
	\end{tikzpicture}
	\caption{A skew $(4,2)$-symplectic tableau of shape $(5,3,1,0 \vert 4,2,1,0)/\allowbreak (2,0 \vert 1,0)$ (left) and its associated family of non-intersecting lattice paths (right).}
	\label{fig:SkewSymplecticGiambelli}
\end{figure}
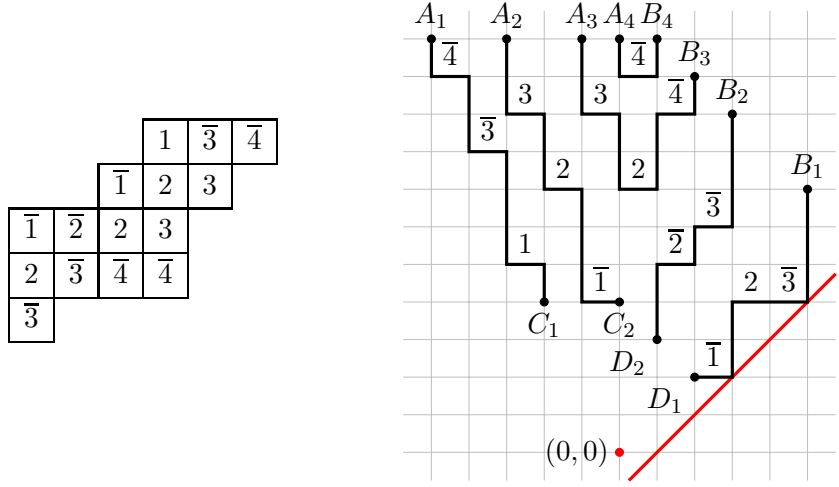
Given a skew shape~$\la/\mu$, the generating function 
$\sp_{\la/\mu}^m(\x)$ is thus the sum over all families of $p+q$ strongly 
non-intersecting lattice paths with starting points $\{ A_1,\dots,A_p,D_1,\dots,D_q\}$ and endpoints $\{ B_1,\dots,B_p,C_1,\dots,C_q\}$
such that $A_i$ is connected to $C_i$ and $D_i$ to $B_i$ for 
$1 \leq i \leq q$ as well as having 
$A_i$ connected to $B_i$ for $q+1 \leq i \leq p$.
Applying the Lindstr\"{o}m--Gessel--Viennot Lemma~\ref{lem:LGV} implies that $\sp^m_{\lambda/\mu}(\x)$ is given by \eqref{eq:GiambelliGesselViennotDet}. There are clearly no lattice paths from any $D_i$ to $C_j$, hence $\mathcal{P} (D_i \to C_j)=0$. Moreover, the paths from $A_i$ to $C_j$ correspond to complete homogeneous symmetric functions, that is, $\mathcal{P} (A_i \to C_j)=h_{\alpha_i-\gamma_j}(\x^\pm)=\sp_{(\alpha_i)/(\gamma_j)}^m(\x)$. Regarding the paths from $D_i$ to $B_j$, we have to keep in mind that no lattice paths are allowed to touch or cross the line $y=x-2$. Hence, Lemma~\ref{lem:SymplecticPaths} implies that $\mathcal{P} (D_i \to B_j) = e_{\beta_j - \delta_i}(\x^\pm) - e_{\beta_j + \delta_i - 2m}(\x^\pm)$. Finally, the paths from $A_i$ to $B_j$ correspond to skew $m$-symplectic characters indexed by hooks, so $\mathcal{P} (A_i \to B_j) = \sp^m_{(\alpha_i \vert \beta_j)}$.

\subsection{Proof of the skew odd orthogonal Giambelli identity~(\ref{Eq_sogiambelli})}

In the case of skew $m$-odd orthogonal characters $\so^m_{\lambda/\mu}$, we consider similar lattice paths as in the previous section with the only addition that we also allow diagonal steps $(1,1)$ along the line $y=x$. These steps correspond to the entries~$\wh{i}$ in the $(m+i)$-th row of the first column of the skew $(n,m)$-odd orthogonal tableau as before and are equipped with the weight~$1$. See Figure~\ref{fig:SkewSpecialOrthogonalGiambelli} for an example.
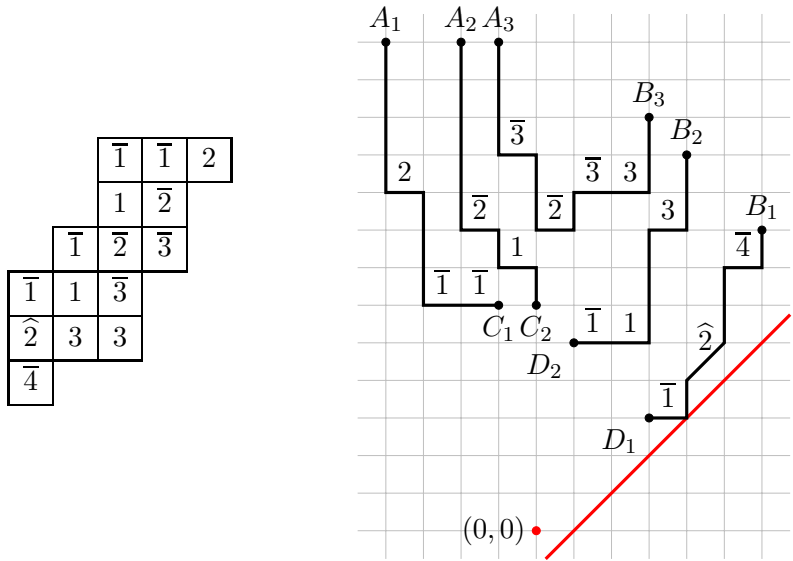
\begin{figure}[htb]
	\centering
	\begin{ytableau}
		\none & \none & \ov{1} & \ov{1} & 2\\
		\none & \none & 1 & \ov{2} \\
		\none & \ov{1} & \ov{2} & \ov{3} \\
		\ov{1} & 1 & \ov{3} \\
		\wh{2} & 3 & 3 \\
		\ov{4} 
	\end{ytableau}
	\qquad\qquad
	\begin{tikzpicture}[scale=.5,baseline=(current bounding box.center)]
		\draw [help lines,step=1cm,lightgray] (-4.75,-.75) grid (6.75,13.75);
		
		\draw[very thick, red] (0.25,-.75) -- (6.75,5.75);
		
		\filldraw[black] (-1,6) circle (3pt);
		\filldraw[black] (0,6) circle (3pt);
		\filldraw[black] (1,5) circle (3pt);
		\filldraw[black] (3,3) circle (3pt);
		\filldraw[black] (-4,13) circle (3pt);
		\filldraw[black] (-2,13) circle (3pt);
		\filldraw[black] (-1,13) circle (3pt);
		\filldraw[black] (3,11) circle (3pt);
		\filldraw[black] (4,10) circle (3pt);
		\filldraw[black] (6,8) circle (3pt);
		
		\draw[very thick] (-4,13) node[above]{$A_1$} -- (-4,9) --node[above]{$2$} (-3,9) -- (-3,6) --node[above]{$\ov{1}$} (-2,6) --node[above]{$\ov{1}$} (-1,6) node[below]{$C_1$};
		\draw[very thick] (-2,13) node[above]{$A_2$} -- (-2,8) --node[above]{$\ov{2}$} (-1,8) -- (-1,7) --node[above]{$1$} (0,7) -- (0,6) node[below]{$C_2$};
		\draw[very thick] (-1,13) node[above]{$A_3$} -- (-1,10) --node[above]{$\ov{3}$} (0,10) -- (0,8) --node[above]{$\ov{2}$} (1,8) -- (1,9) --node[above]{$\ov{3}$} (2,9) --node[above]{$3$} (3,9) -- (3,11) node[above]{$B_3$};
		\draw[very thick] (1,5) node[below left]{$D_2$} --node[above]{$\ov{1}$} (2,5) --node[above]{$1$} (3,5) -- (3,8) --node[above]{$3$} (4,8) -- (4,10) node[above]{$B_2$};
		\draw[very thick] (3,3) node[below left]{$D_1$} --node[above]{$\ov{1}$} (4,3) -- (4,4) --node[above]{$\wh{2}$} (5,5) -- (5,7) --node[above]{$\ov{4}$} (6,7) -- (6,8) node[above]{$B_1$};
		
		\filldraw[red] (0,0) circle (3pt) node[left]{\color{black}$(0,0)$};
	\end{tikzpicture}
	\caption{A skew $(4,3)$-odd orthogonal tableau of shape $(4,2,1 \vert 5,3,2)/\allowbreak (1,0 \vert 2,0)$ (left) and its associated family of non-intersecting lattice paths (right).}
	\label{fig:SkewSpecialOrthogonalGiambelli}
\end{figure}
	As seen before, it suffices to compute the determinantal expression~\eqref{eq:GiambelliGesselViennotDet} in our setting. The expressions for $\mathcal{P} (D_i \to C_j)$ and $\mathcal{P} (A_i \to C_j)$ are the same as in the symplectic case, so $\mathcal{P} (D_i \to C_j) = 0$ and $\mathcal{P} (A_i \to C_j) = h_{\alpha_i - \gamma_j}(\x^\pm)=\so_{(\alpha_i)/(\gamma_j)}^m(\x)$. The entry $\mathcal{P} (D_i \to B_j) = e_{\beta_j - \delta_i}(\x^\pm) + e_{\beta_j + \delta_i -2m+1}(\x^\pm)=\so_{(1^{\beta_j+1})/(1^{\delta_i+1})}^m(\x)$ is a simple consequence of Corollary~\ref{cor:SpecialOrthogonalPaths}, and $\mathcal{P} (A_i \to B_j)$ corresponds to the skew $m$-odd orthogonal character $\so^m_{( \alpha_i \vert \beta_j)}(\x)$.

\subsection{Proof of the skew even orthogonal Giambelli identity~(\ref{Eq_ogiambelli})}

Now, we consider lattice paths in our general set-up with step set $\{ (1,0),(0,-1)\}$ in the region $x \leq 0$ and step set $\{ (1,0),(0,1),(2,0)\}$ in the region $x \geq 0$ such that all lattice paths stay weakly above the line $y=x-1$ and $\o$-horizontal steps are only allowed if they end on the line $y=x$. As before, we interpret the entries of a skew $(n,m)$-even orthogonal tableau of shape $\lambda/\mu$ hookwise as lattice paths. Yet again, consecutive entries $\wc{i}$ and $\wh{i}$ in the first column are interpreted as an $\o$-horizontal step with weight~$1$ drawn as an arch. As a result, we obtain a family of $p+q$ strongly non-intersecting lattice paths, where $\lambda=(\alpha_1,\dots,\alpha_p \vert \beta_1,\dots,\beta_p)$ and $\mu=(\gamma_1,\dots,\gamma_q \vert \delta_1,\dots,\delta_q)$. See Figure~\ref{fig:SkewOrthogonalGiambelli} for an example.
\begin{figure}[htb]
	\centering
	\begin{ytableau}
		\none & \none & \none & \none & \ov{1} \\
		\none & \none & 1 & 1 & 2 \\
		\wc{2} & \ov{2} & 2 & \ov{3} & \ov{4} \\
		\wh{2} & \ov{3} & \ov{3} & 3 \\
		\ov{3} & 3 & \ov{4} & 4 \\
		\ov{5} & 5 \\
		5 
	\end{ytableau}
	\qquad\qquad
	\begin{tikzpicture}[scale=.5,baseline=(current bounding box.center)]
		\draw [help lines,step=1cm,lightgray] (-4.75,-.75) grid (7.75,13.75);
		
		\draw[very thick, red] (0.25,-.75) -- (7.75,6.75);
		
		\filldraw[black] (-3,4) circle (3pt);
		\filldraw[black] (0,4) circle (3pt);
		\filldraw[black] (2,2) circle (3pt);
		\filldraw[black] (1,3) circle (3pt);
		\filldraw[black] (7,7) circle (3pt);
		\filldraw[black] (5,9) circle (3pt);
		\filldraw[black] (3,11) circle (3pt);
		\filldraw[black] (2,12) circle (3pt);
		\filldraw[black] (0,13) circle (3pt);
		\filldraw[black] (-2,13) circle (3pt);
		\filldraw[black] (-3,13) circle (3pt);
		\filldraw[black] (-4,13) circle (3pt);

		\draw[very thick] (-4,13) node[above]{$A_1$} -- (-4,4) --node[above]{$\ov{1}$} (-3,4) node[below]{$C_1$};
		\draw[very thick] (-3,13) node[above]{$A_2$} -- (-3,7) --node[above]{$2$} (-2,7) -- (-2,5) --node[above]{$1$} (-1,5) --node[above]{$1$} (0,5) -- (0,4) node[below]{$C_2$};
		\draw[very thick] (-2,13) node[above]{$A_3$} -- (-2,10) --node[above]{$\ov{4}$} (-1,10) -- (-1,8) --node[above]{$\ov{3}$} (0,8) -- (0,7) --node[above]{$2$} (1,7) --node[above]{$\ov{3}$} (2,7) -- (2,8) --node[above]{$\ov{4}$} (3,8) -- (3,11) node[above]{$B_3$};
		\draw[very thick] (0,13) node[above]{$A_4$} -- (0,9) --node[above]{$3$} (1,9) -- (1,10) --node[above]{$4$} (2,10) -- (2,12) node[above]{$B_4$};
		\draw[very thick] (1,3) node[below left]{$D_2$} -- (1,5) --node[above]{$\ov{2}$} (2,5) -- (2,6) --node[above]{$\ov{3}$} (3,6) --node[above]{$3$} (4,6) -- (4,9) --node[above]{$5$} (5,9) node[right]{$B_2$};
		\draw[very thick] (2,2) node[below left]{$D_1$} -- (2,4) to[out=45, in=135] (4,4) --node[above]{$\ov{3}$} (5,4) -- (5,7) --node[above]{$\ov{5}$} (6,7) --node[above]{$5$} (7,7) node[above right]{$B_1$};
		\node at (2.5,5) {$\wc{2}$};
		\node at (3.5,5) {$\wh{2}$};
		\filldraw[red] (0,0) circle (3pt) node[left]{\color{black}$(0,0)$};
	\end{tikzpicture}
	\caption{A skew $(5,2)$-even orthogonal tableau of shape $(4,3,2,0 \vert \allowbreak 6,4,2,1)/\allowbreak (3,0 \vert 1,0)$ (left) and its associated family of non-intersecting lattice paths (right).}
	\label{fig:SkewOrthogonalGiambelli}
\end{figure}
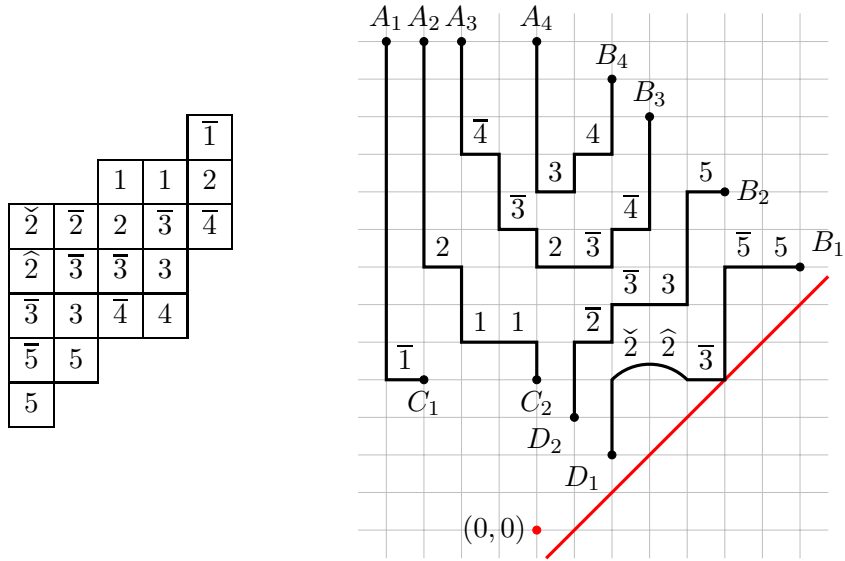
However, not every family of non-intersecting lattice paths in our set-up corresponds to a skew $(n,m)$-even orthogonal tableau. In fact, the $m$-even orthogonal condition implies that the configurations illustrated in Figure~\ref{fig:ForbiddenConfigOrthogonalGiambelli} cannot occur. To be more precise, assume we have a skew $(n,m)$-even orthogonal tableau with entry $\ov{i}$ in the first column of the $(m+i)$-th row and $i$ also appears in the $j$-th column of that row. Then the entry in the same column one row above is supposed to be $\ov{i}$. If that entry at position $(m+i-1,j)$ is not $\ov{i}$, then we obtain one of the three cases in Figure~\ref{fig:ForbiddenConfigOrthogonalGiambelli}; namely (\subref{subfig:config_eside}) if $j<m+i$, (\subref{subfig:config_mixed_ehside}) if $j=m+i$ and (\subref{subfig:config_hside}) if $j>m+i$. In particular, the vacancies indicated by the red points imply that the families of lattice paths are not obtained by skew $(n,m)$-even orthogonal tableaux. Note that case (\subref{subfig:config_eside}) is equivalent to the case of the trapped position in Section~\ref{sec:oYT}. The reason for the other cases is that the trapped position moves from a region with $e$-labelling to a region with $h$-labelling.

The corners of the nested left turns in Figure~\ref{fig:ForbiddenConfigOrthogonalGiambelli} are each lying on the line $y+x=2m+2i-1$; we say that $2m+2i-1$ is the distance between the origin and the trapped positions.

\begin{figure}[htb]
	\centering
	\subcaptionbox{\label{subfig:config_eside}}{
		\begin{tikzpicture}[scale=.5,baseline=(current bounding box.center)]
			\draw [help lines,step=1cm,lightgray] (-1.75,-.75) grid (-.5,4.75);
			\draw [help lines,step=1cm,lightgray] (.5,-.75) grid (5.75,4.75);
			\draw[->,ultra thin] (-1,-.75)--(-1,4.75);
			\draw[very thick, red] (4.25,-.75) -- (5.75,.75); 
			
			\draw[very thick] (4,0) -- (5,0) -- (5,1);
			
			\filldraw[black] (4,1) circle (1pt);
			\filldraw[black] (4.25,.75) circle (1pt);
			\filldraw[black] (3.75,1.25) circle (1pt);
			
			\filldraw[black] (-.25,2) circle (1pt);
			\filldraw[black] (0,2) circle (1pt);
			\filldraw[black] (.25,2) circle (1pt);
			
			\draw[very thick] (2,2) -- (3,2) -- (3,3);
			
			\draw[very thick] (1,3) -- (1,4) -- (2,4);
			
			\filldraw[red] (2,3) circle (3pt);
	\end{tikzpicture}}
	\hfill
	\subcaptionbox{\label{subfig:config_mixed_ehside}}{
		\begin{tikzpicture}[scale=.5,baseline=(current bounding box.center)]
			\draw [help lines,step=1cm,lightgray] (.25,-.75) grid (5.75,4.75);
			\draw[->] (1,-.75)--(1,4.75);
			\draw[very thick,red] (4.25,-.75) -- (5.75,.75); 
			
			\draw[very thick] (4,0) -- (5,0) -- (5,1);
			
			\filldraw[black] (4,1) circle (1pt);
			\filldraw[black] (4.25,.75) circle (1pt);
			\filldraw[black] (3.75,1.25) circle (1pt);
			
			\draw[very thick] (2,2) -- (3,2) -- (3,3);
			
			\draw[very thick] (1,4) -- (2,4);
			
			\filldraw[red] (1,3) circle (3pt);
			\filldraw[red] (2,3) circle (3pt);
	\end{tikzpicture}}
	\hfill
	\subcaptionbox{\label{subfig:config_hside}}{
		\begin{tikzpicture}[scale=.5,baseline=(current bounding box.center)]
			\draw [help lines,step=1cm,lightgray] (-4.75,-.75) grid (-2.5,4.75);
			\draw [help lines,step=1cm,lightgray] (-1.5,-.75) grid (-.5,4.75);
			\draw [help lines,step=1cm,lightgray] (.5,-.75) grid (5.75,4.75);
			\draw[->] (-1,-.75)--(-1,4.75);
			\draw[very thick, red] (4.25,-.75) -- (5.75,.75); 
			
			\draw[very thick] (4,0) -- (5,0) -- (5,1);
			
			\filldraw[black] (-2.25,2) circle (1pt);
			\filldraw[black] (-2,2) circle (1pt);
			\filldraw[black] (-1.75,2) circle (1pt);
			
			\filldraw[black] (-.25,2) circle (1pt);
			\filldraw[black] (0,2) circle (1pt);
			\filldraw[black] (.25,2) circle (1pt);
			
			\filldraw[black] (4,1) circle (1pt);
			\filldraw[black] (4.25,.75) circle (1pt);
			\filldraw[black] (3.75,1.25) circle (1pt);
			
			\draw[very thick] (2,2) -- (3,2) -- (3,3);
			
			\draw[very thick] (.25,3) -- (1,3) -- (2,3) -- (2,4);	
			
			\draw[very thick]  (-4,4) -- (-3,4) -- (-3,3) --++ (0.75,0);
			\draw[very thick] (-1.75,3) -- (-.25,3);
			
			\filldraw[red] (-4,3) circle (3pt);
	\end{tikzpicture}}
	\caption{Local configurations of trapped positions at $(m+i-1,j)$ if (a) $j<m+i$, (b) $j=m+1$ or (c) $j>m+1$. The large points indicate vacancies.}
	\label{fig:ForbiddenConfigOrthogonalGiambelli}
\end{figure}
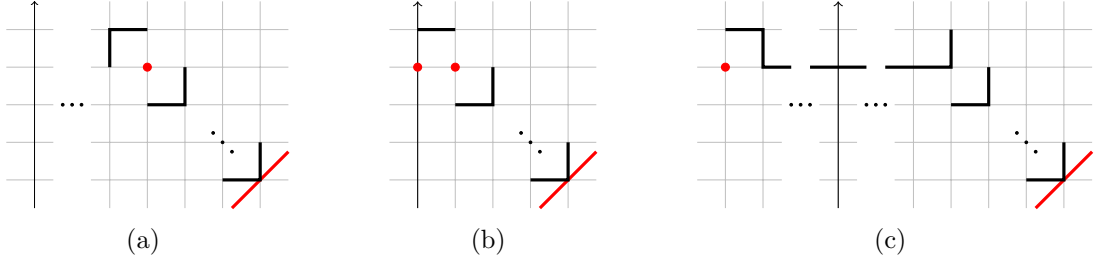

As in the cases before, we want to apply the Lindstr\"om--Gessel--Viennot Lemma~\ref{lem:LGV}. The evaluation of the determinant~\eqref{eq:GiambelliGesselViennotDet} in the current set-up yields
\[
	\frac{(-1)^q}{2^{[m=l(\mu)\wedge m\neq0]}} \det \begin{pmatrix}
		( \o^m_{( \alpha_i \vert \beta_j)} (\x))_{1 \leq i,j \leq p} & ( h_{\alpha_i - \gamma_j} (\x^\pm))_{\substack{1 \leq i \leq p,\\1 \leq j \leq q}}\\
		( e_{\beta_j - \delta_i} (\x^\pm) + e_{\beta_j + \delta_i -2m+2} (\x^\pm))_{\substack{1 \leq i \leq q,\\1 \leq j \leq p}} & 0
	\end{pmatrix},
\]
which, by rewriting the entries in terms of the even orthogonal characters, is equal to \eqref{Eq_ogiambelli}.
However, a priori this is not the generating function of skew $(n,m)$-even orthogonal tableaux since it also enumerates the trapped positions in Figure~\ref{fig:ForbiddenConfigOrthogonalGiambelli} as well as families of weakly but not strongly non-intersecting lattice paths with intersections of $\o$-horizontal steps with vertical steps along the line $y=x+1$. We provide a sign-reversing involution under which these families of lattice paths cancel out, ultimately showing that the above is indeed equal to $\o^m_{\lambda/\mu}(\x)$.

	Consider the families of lattice paths enumerated by \eqref{Eq_ogiambelli} as described above. Assume a given family of lattice paths has intersections involving $\o$-horizontal steps or trapped positions that we have specified above. If an intersection lies on the line $y+x=d$, then we say that $d$ is the distance between the origin and that intersection. We take the unique trapped position or intersection with the smallest distance to the origin and perform the corresponding local changes shown in Figure~\ref{fig:OrthogonalPathsLocalChanges_Giambelli}; the rest of the paths are left unchanged. Thus, the weight of the family of lattice paths is exactly changed by a factor of $-1$.
	
	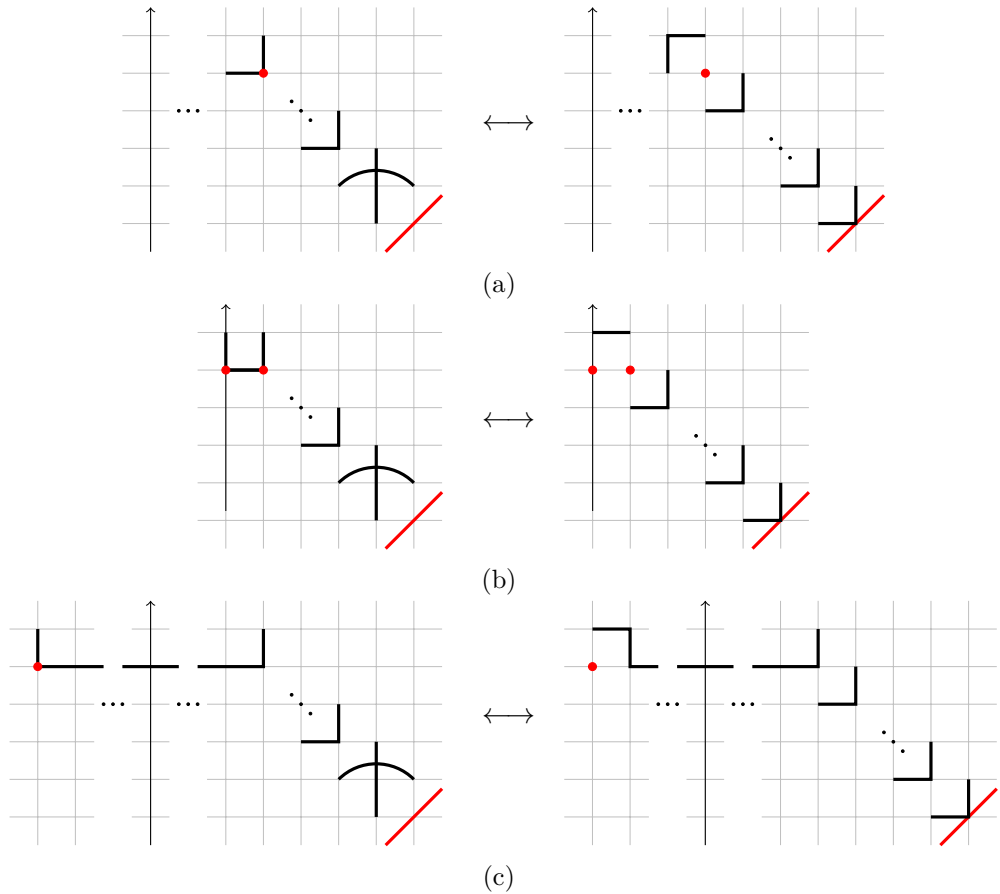
\begin{figure}[htb]
		\centering
		\subcaptionbox{\label{subfig:LocalChanges_eside}}{
			\begin{tikzpicture}[scale=.5,baseline=(current bounding box.center)]
				\draw [help lines,step=1cm,lightgray] (-1.75,-1.75) grid (-.5,4.75);
				\draw [help lines,step=1cm,lightgray] (.5,-1.75) grid (6.75,4.75);
				\draw[->] (-1,-1.75)--(-1,4.75);
				\draw[very thick, red] (5.25,-1.75) -- (6.75,-.25); 
				
				\filldraw[black] (3,2) circle (1pt);
				\filldraw[black] (3.25,1.75) circle (1pt);
				\filldraw[black] (2.75,2.25) circle (1pt);
				
				\filldraw[black] (-.25,2) circle (1pt);
				\filldraw[black] (0,2) circle (1pt);
				\filldraw[black] (.25,2) circle (1pt);
				
				\draw[very thick] (1,3) -- (2,3) -- (2,4);
				\draw[very thick] (3,1) -- (4,1) -- (4,2);
				
				\draw[very thick] (4,0) to[out=45, in=135] (6,0);
				
				\draw[very thick] (5,-1) -- (5,1);
				\filldraw[red] (2,3) circle (3pt);
			\end{tikzpicture}
			\quad$\longleftrightarrow$\quad
			\begin{tikzpicture}[scale=.5,baseline=(current bounding box.center)]
				\draw [help lines,step=1cm,lightgray] (-1.75,-1.75) grid (-.5,4.75);
				\draw [help lines,step=1cm,lightgray] (.5,-1.75) grid (6.75,4.75);
				\draw[->] (-1,-1.75)--(-1,4.75);
				\draw[very thick, red] (5.25,-1.75) -- (6.75,-.25); 
				
				\draw[very thick] (4,0) -- (5,0) -- (5,1);
				
				\filldraw[black] (4,1) circle (1pt);
				\filldraw[black] (4.25,.75) circle (1pt);
				\filldraw[black] (3.75,1.25) circle (1pt);
				
				\filldraw[black] (-.25,2) circle (1pt);
				\filldraw[black] (0,2) circle (1pt);
				\filldraw[black] (.25,2) circle (1pt);
				
				\draw[very thick] (2,2) -- (3,2) -- (3,3);
				
				\draw[very thick] (1,3) -- (1,4) -- (2,4);
				\draw[very thick] (5,-1) -- (6,-1) -- (6,0);				
				
				\filldraw[red] (2,3) circle (3pt);
		\end{tikzpicture}}
		
		\subcaptionbox{\label{subfig:LocalChanges_mixed_ehside}}{
			\begin{tikzpicture}[scale=.5,baseline=(current bounding box.center)]
				\draw [help lines,step=1cm,lightgray] (.25,-1.75) grid (6.75,4.75);
				\draw[->] (1,-.75)--(1,4.75);
				\draw[very thick, red] (5.25,-1.75) -- (6.75,-.25); 
				
				\filldraw[black] (3,2) circle (1pt);
				\filldraw[black] (3.25,1.75) circle (1pt);
				\filldraw[black] (2.75,2.25) circle (1pt);
				
				\draw[very thick] (1,3) -- (2,3) -- (2,4);
				\draw[very thick] (3,1) -- (4,1) -- (4,2);
				
				\draw[very thick] (4,0) to[out=45, in=135] (6,0);
				
				\draw[very thick] (5,-1) -- (5,1);
				
				\draw[very thick] (1,4) -- (1,3) -- (2,3) -- (2,4);
				\filldraw[red] (1,3) circle (3pt);
				\filldraw[red] (2,3) circle (3pt);
			\end{tikzpicture}
			\quad$\longleftrightarrow$\quad
			\begin{tikzpicture}[scale=.5,baseline=(current bounding box.center)]
				\draw [help lines,step=1cm,lightgray] (.25,-1.75) grid (6.75,4.75);
				\draw[->] (1,-.75)--(1,4.75);
				\draw[very thick, red] (5.25,-1.75) -- (6.75,-.25); 
				
				\draw[very thick] (4,0) -- (5,0) -- (5,1);
				
				\filldraw[black] (4,1) circle (1pt);
				\filldraw[black] (4.25,.75) circle (1pt);
				\filldraw[black] (3.75,1.25) circle (1pt);
				
				\draw[very thick] (2,2) -- (3,2) -- (3,3);
				
				\draw[very thick] (1,4) -- (2,4);
				\draw[very thick] (5,-1) -- (6,-1) -- (6,0);	
				\filldraw[red] (1,3) circle (3pt);
				\filldraw[red] (2,3) circle (3pt);
		\end{tikzpicture}}
		
		\subcaptionbox{\label{subfig:LocalChanges_hside}}{
			\begin{tikzpicture}[scale=.5,baseline=(current bounding box.center)]
				\draw [help lines,step=1cm,lightgray] (-4.75,-1.75) grid (-2.5,4.75);
				\draw [help lines,step=1cm,lightgray] (-1.5,-1.75) grid (-.5,4.75);
				\draw [help lines,step=1cm,lightgray] (.5,-1.75) grid (6.75,4.75);
				\draw[->] (-1,-1.75)--(-1,4.75);
				\draw[very thick, red] (5.25,-1.75) -- (6.75,-.25); 
				
				\filldraw[black] (-2.25,2) circle (1pt);
				\filldraw[black] (-2,2) circle (1pt);
				\filldraw[black] (-1.75,2) circle (1pt);
				
				\filldraw[black] (-.25,2) circle (1pt);
				\filldraw[black] (0,2) circle (1pt);
				\filldraw[black] (.25,2) circle (1pt);
				
				\filldraw[black] (3,2) circle (1pt);
				\filldraw[black] (3.25,1.75) circle (1pt);
				\filldraw[black] (2.75,2.25) circle (1pt);
				
				\draw[very thick] (.25,3) -- (1,3) -- (2,3) -- (2,4);	
				
				\draw[very thick]  (-4,4) -- (-4,3) -- (-3,3) --++ (0.75,0);
				\draw[very thick] (-1.75,3) -- (-.25,3);
				
				\draw[very thick] (3,1) -- (4,1) -- (4,2);
				
				\draw[very thick] (4,0) to[out=45, in=135] (6,0);
				
				\draw[very thick] (5,-1) -- (5,1);
				\filldraw[red] (-4,3) circle (3pt);
			\end{tikzpicture}
			\quad$\longleftrightarrow$\quad
			\begin{tikzpicture}[scale=.5,baseline=(current bounding box.center)]
				\draw [help lines,step=1cm,lightgray] (-4.75,-1.75) grid (-2.5,4.75);
				\draw [help lines,step=1cm,lightgray] (-1.5,-1.75) grid (-.5,4.75);
				\draw [help lines,step=1cm,lightgray] (.5,-1.75) grid (6.75,4.75);
				\draw[->] (-1,-1.75)--(-1,4.75);
				\draw[very thick, red] (5.25,-1.75) -- (6.75,-.25); 
				
				\draw[very thick] (4,0) -- (5,0) -- (5,1);
				
				\filldraw[black] (-2.25,2) circle (1pt);
				\filldraw[black] (-2,2) circle (1pt);
				\filldraw[black] (-1.75,2) circle (1pt);
				
				\filldraw[black] (-.25,2) circle (1pt);
				\filldraw[black] (0,2) circle (1pt);
				\filldraw[black] (.25,2) circle (1pt);
				
				\filldraw[black] (4,1) circle (1pt);
				\filldraw[black] (4.25,.75) circle (1pt);
				\filldraw[black] (3.75,1.25) circle (1pt);
				
				\draw[very thick] (2,2) -- (3,2) -- (3,3);
				
				\draw[very thick] (.25,3) -- (1,3) -- (2,3) -- (2,4);	
				
				\draw[very thick]  (-4,4) -- (-3,4) -- (-3,3) --++ (0.75,0);
				\draw[very thick] (-1.75,3) -- (-.25,3);
				
				\draw[very thick] (5,-1) -- (6,-1) -- (6,0);
				\filldraw[red] (-4,3) circle (3pt);
		\end{tikzpicture}}
		\caption{Local changes on the families of lattice paths enumerated by \eqref{Eq_ogiambelli}.}
		\label{fig:OrthogonalPathsLocalChanges_Giambelli}
	\end{figure}
	
	This sign-reversing involution cancels all families of lattice paths that contain one of the trapped positions or an intersection of an $\o$-horizontal step with two vertical steps. We are left with exactly those families of strongly non-intersecting lattice paths that correspond to skew $(n,m)$-even orthogonal tableaux. This completes the proof of Theorem~\ref{thm:giambelli}.

\subsection*{Acknowledgement} 
We thank Ole Warnaar for useful comments on an earlier version of this paper.

\bibliographystyle{amsplain}
\bibliography{skews}

\end{document}